\documentclass{article}

\usepackage{mathtools}
\usepackage{graphicx}
\usepackage{amsfonts}
\usepackage{amssymb}
\usepackage{amsmath}
\usepackage{amsthm}
\usepackage{url}
\usepackage{color}
\usepackage{multirow}
\usepackage{enumerate}
\usepackage{hyperref}
\usepackage{epstopdf}
\usepackage{tikz}
\usepackage{here}
\usepackage{bigstrut}
\usepackage{pgfplots}
\usepackage{subcaption}
\usepackage{algorithm}
\usepackage{algpseudocode}
\usepackage{adjustbox}
\usepackage{xcolor, colortbl}
\usepackage[dvipsnames]{xcolor}
\usepackage[indent]{parskip}
\usepackage{afterpage}
\usepackage{tcolorbox}
\usepackage{comment}
\usepackage{booktabs}
\usepackage[normalem]{ulem}
\usepackage{pdflscape}
\usepackage{graphicx}
\usepackage{subcaption}   
\usepackage{float} 

\DeclareMathOperator{\diff}{diff}
\DeclareMathOperator{\Diag}{Diag}

\newtheorem{thm}{Theorem}
\newtheorem{lem}[thm]{Lemma}
\newtheorem{cor}[thm]{Corollary}
\newtheorem{prop}[thm]{Proposition}

\colorlet{mypurple}{purple!20}

\pgfplotsset{compat=1.18}

\title{Quantum and Simulated Annealing-Based Iterative Algorithms for QUBO Relaxations of the  Sparsest 
$k$-Subgraph Problem}
\author{
	{Omkar Bihani} \thanks{Rudolfovo, Science and Technology Centre, Novo mesto, Slovenia, {\tt omkar.bihani@rudolfovo.eu}}
\and
	{Roman Kužel} \thanks{Rudolfovo, Science and Technology Centre, Novo mesto, Slovenia, {\tt roman.kuzel@rudolfovo.eu}}
\and {Janez Povh}\thanks{Rudolfovo, Science and Technology Centre, Novo mesto, Slovenia, {\tt janez.povh@rudolfovo.eu}}
        \and  {Dunja Pucher}\thanks{Department of Mathematics,
		University of Klagenfurt, Austria,
		{\tt dunja.pucher@aau.at} }}
\date{\today}

\begin{document}

\maketitle

% \tableofcontents
% \newpage

\begin{abstract} 
In this paper, we introduce three QUBO (Quadratic Unconstrained Binary Optimization) relaxations for the sparsest 
$k$-subgraph (SkS) problem: a quadratic penalty relaxation, a Lagrangian relaxation, and an augmented Lagrangian relaxation. The effectiveness of these approaches strongly depends on the choice of penalty parameters. We establish theoretical results characterizing the parameter values for which the QUBO relaxations are exact. For practical implementation, we propose three iterative algorithms, which have in their kernel the QUBO relaxations, that update the penalty parameters at each iteration while approximately solving the internal QUBO problems with simulated annealing and quantum processing units. Extensive numerical experiments validate our theoretical findings on exact relaxations and demonstrate the efficiency of the proposed iterative algorithms.    
\end{abstract}

	\\[1mm]
	{\small\noindent\textbf{Math. Subj. Class. (2020):} 90C27, 81P68}
	
\section{Introduction}

\subsection{Motivation and related work}

The sparsest $k$-subgraph problem (SkS) is a fundamental problem in combinatorial optimization. Given a simple and undirected graph $G = (V, E)$ on $\vert V \vert = n$ vertices, and a positive integer $k \leq n$, the goal is to identify a subset $S \subseteq V$ on exactly $k$ vertices such that the number of edges in the subgraph induced by $S$ is minimized. In its weighted variant, also known as the $k$-lightest subgraph problem, the aim is to minimize the total weight of the induced edges. 

The SkS problem generalizes the classical maximum stable set problem, where the induced subgraph is required to contain no edges at all. As a result, the SkS problem inherits the computational hardness of the maximum set problem and is NP-hard for general graphs; see, for instance,~\cite{watrigant}.

Besides its connection to the maximum stable set problem, the SkS problem is also related to the densest $k$-subgraph problem via complement graphs. It therefore offers a rich framework for investigating combinatorial structures and optimization techniques.

Several classical approaches have been developed for solving the sparsest or the densest $k$-subgraph problem. These methods include exact algorithms based on integer programming, heuristics such as greedy and local search techniques, and relaxation-based approaches using semidefinite programming. For a detailed overview of classical methods and their applications, we refer the reader to the recent survey~\cite{Lanciano2024}.

The SkS problem drew our attention for another reason: it can be naturally formulated as a Quadratic Unconstrained Binary Optimization (QUBO) problem with a single linear constraint. QUBO problems have recently received considerable attention because they can be efficiently solved, at least in theory, by using quantum annealers such as those provided by D-Wave systems. Our initial motivation, therefore, was to transform the single linear constraint into the objective function and evaluate how quantum annealers perform on the resulting  QUBO formulation.

 We investigate three approaches that result in three different QUBO  relaxations for the SkS problem. We first apply the most straightforward approach-the quadratic penalty approach. Inspired by the recent results in~\cite{Djidjev2023}, we also apply the augmented Lagrangian method to represent the cardinality constraint more effectively, and as the third approach, we utilize Lagrangian relaxation in a manner different from that demonstrated in~\cite{Ronah}. 

Quantum annealers such as D-Wave Quantum Processing Unit (QPU) are very sensitive to the magnitude and scaling of the coefficients in the objective function, since they normalize all the coefficients in the objective function to meet the device restrictions. Consequently, the inclusion of linear constraints into the objective function using a penalisation technique raises an important question: how should the penalty parameters be chosen to avoid degrading the performance of the solver?

 We show that the quadratic penalty approach influences all the coefficients of the original objective function, so a large penalty parameter decreases the value of the original coefficients after normalization,  which, combined with the finite precision of the digital-to-analog converter,  reduces the accuracy of quantum annealing; see, for instance,~\cite{Djidjev2023}. In contrast, if this parameter is too small, infeasible solutions may dominate, leading to incorrect results. Glover et al. ~\cite{Glover2019} claim that the penalty parameters are not unique, and considering a range of values often provides effective results.

The quantum annealing approach to the augmented Lagrangian method, where the standard penalty method is combined with the Lagrangian multipliers, was studied by Djidjev in  \cite{Djidjev2023_2, Djidjev2023}. In the case of SkS, we need to introduce two parameters: one for the linear and one for the quadratic term. The Djidjev approach employs an iterative procedure to determine optimal or near-optimal values for both penalty parameters. 
The augmented Lagrangian method, combined with quantum annealing, partially resolves the  challenges related to the reduced precision when considering large-scale problems. This method was initially developed for continuous variables, as exposed in~\cite{Djidjev2023_2}, and for optimization problems in discrete variables, there are no known convergence  results. Nevertheless, the experimental results in~\cite{Djidjev2023} for the set cover problem show that the augmented Lagrangian approach in combination with quantum annealing outperforms the standard penalty method,which introduces slack variables and quadratic penalty.

In the third approach, we omit the quadratic penalty term entirely and instead consider the Lagrangian relaxation of the SkS.
A detailed study of Lagrangian relaxation for continuous optimization problems can be found in~\cite{Ahuja1993},  its application to integer programming is explored in~\cite{Fisher2004}, while a method for solving the Lagrangian dual of a constrained binary quadratic programming problem using a quantum annealer is presented in~\cite{Ronah}.

The challenges with selecting the penalty parameters can be overcome by sequential penalty methods; see, for instance,~\cite{Garcia2022}. For combinatorial optimization problems,  which have a specific structure and characteristics, this method can yield more efficient and effective results. Several illustrative examples of such tailored approaches are discussed in~\cite{Quintero2020}.

The challenges with selecting the penalty parameters can be overcome by sequential penalty methods; see, for instance,~\cite{Garcia2022}. For combinatorial optimization problems,  which have a specific structure and characteristics, this method can yield more efficient and effective results. Several illustrative examples of such tailored approaches are discussed in~\cite{Quintero2020}.

Building on these insights, this paper incorporates all major approaches from the literature and extends them further. For each of the three penalty approaches to the SkS problem, it develops an iterative algorithm that adaptively determines appropriate penalty values, ensuring feasibility without significantly compromising the solver performance. 

To the best of our knowledge, this is the first study that applies different penalty-based approaches in combination with quantum annealers to the SkS problem. 
Prior work on related problems is limited. One study from 2020~\cite{Calude2020} applied quantum annealing to the heaviest $k$-subgraph problem, the weighted variant of the densest $k$-subgraph problem, using random bipartite graphs with 30 vertices. Another related work is a patented method from 2024~\cite{Ronah}, which presents computational results for a toy example of the heaviest $k$-subgraph problem by solving the Lagrangian dual of a constrained binary quadratic programming problem using a quantum annealer. However, a comprehensive computational study on different instances is not provided.

\subsection{Our contribution and outline}

This article presents a comprehensive investigation of three  Quadratic Unconstrained Binary Optimization (QUBO) relaxations for the sparsest $k$-subgraph (SkS) problem and algorithms to solve them. Our main contributions are as follows:

\begin{itemize}
   \item We provide a detailed theoretical framework for three distinct QUBO relaxations of the SkS problem: the  Quadratic Penalty (QP), Lagrangian Relaxation (LR), and Augmented Lagrangian (AL) relaxations. For each, we establish conditions and derive bounds for their respective parameters (quadratic penalty parameter $\mu$ and Lagrange parameter $\lambda$) that guarantee the exactness, i.e., the optimal solutions of the QUBO relaxations are optimal solutions of the SkS problem. 

    \item We delve into the nuances of Lagrangian Relaxation for the SkS problem, particularly addressing the challenge of selecting an appropriate Lagrange multiplier $\lambda$.
 We demonstrate that for specific graph properties, such as when the sequence $\{\diff_{\ell}\}$
     (representing the difference in minimum edges between subgraphs of consecutive sizes) is monotonically increasing, exact solutions can be obtained with a precisely determined $\lambda$.
     \item We provide extensive numerical results to demonstrate the exactness of the three relaxations using the exact QUBO solver BiqBin \cite{gusmeroli2022biqbin}.  
     
     \item We develop and implement three iterative algorithms -- Quadratic Penalty Iterative Algorithm (QPIA), Lagrangian Relaxation Iterative Algorithm (LRIA), and Augmented Lagrangian Iterative Algorithm (ALIA), designed to address the computational challenges of solving SkS for larger instances. These algorithms dynamically adjust quadratic penalty and Lagrangian parameters and use the D-Wave quantum annealing and  Simulated Annealing (SA) solvers in each iteration.

     \item Through extensive numerical experiments on various graph datasets, including Erd\H{o}s--R\'enyi (ER) graphs, Bipartite graphs, and D-Wave topology graphs, we empirically validate the effectiveness of our proposed methods.
     
     \end{itemize}

 \subsection{Terminology and notation}

We close this section with some terminology and notation. The vector of all-ones of length $n$ is denoted by $e_n$.
% , and we write $I_n$ and $J_n$ for the identity and the all-ones matrix of order $n$, respectively.
If $n$ is clear from the context, we write shortly $e$.
% , $I$, and $J$. 
Furthermore, if $a \in \mathbb{R}^n$, then $\Diag(a)$ denotes a $n \times n$ diagonal matrix with $a$ on the main diagonal.

Throughout this paper, $G = (V(G), E(G))$ denotes a simple undirected graph with $\vert V(G) \vert = n$ vertices and $\vert E(G) \vert = m$ edges. Without loss of generality, we assume that the set of vertices is $V(G) = \{1, \ldots, n\}$. If a graph $G$ is clear from the context, we write shortly $G = (V, E)$. The complement of a graph $G = (V, E)$ is the graph $\overline{G} = (V, \overline{E})$, where $\overline{E}$ is the set of non-edges in $G$, that is $\overline{E} = \{ \{i, j\} \in V\; \vert \;  \{i,j\} \notin E, ~i \neq j\}$. A graph $G' = (V', E') $ is a subgraph of $G$ if $V' \subseteq V$ and $E' \subseteq E$. Let $V' \subseteq V$. We say that the subgraph $G' = ( V', E')$ of $G$ is induced by $V'$ if $E'$ consists of all edges $\{i,j\}$ from $E$  with $i,j \in V'$.  
For a vertex $i\in V(G)$, we denote its degree by $d_G(i)$. The number~$\Delta(G) = \max \{d_G(i) \colon i \in V(G) \}$ is the maximum degree of $G$.

\section{Constrained and unconstrained quadratic binary optimization problems}\label{QUBO_theory}

The sparsest $k$-subgraph problem, which is formally defined in the next section, is an example of a quadratic binary optimization problem with linear constraints. Such problems can be relaxed to quadratic unconstrained binary optimization problems using various penalization techniques.
In this section, we recall definitions and the most important results related to these relaxations, as well as some associated challenges.

Quadratic unconstrained binary optimization (QUBO) is a class of optimization problems formally defined as:
\begin{align}\label{general_QUBO}
\tag{\mbox{QUBO}}
\min \left \{x^T Q x + c^T x \mid x \in \{0, 1\}^n \right \},
\end{align}
where $Q \in \mathbb{R}^{n \times n}$ is a symmetric matrix and $c \in \mathbb{R}^n$. Since $x_i^2 = x_i$ for all $i \in \{1, \dots, n\}$, the expression $x^T Q x + c^T x$ can be rewritten as $x^T (Q + \Diag(c)) x$.
Therefore, without loss of generality, a QUBO problem can be expressed with a purely quadratic objective function, omitting the explicit linear term.

In general, the QUBO problem is NP-hard, due to its equivalence to the Max-Cut problem, for which it is proven that its decision version is an NP-complete problem - it is on the list of NP-complete problems from   \cite{garey1979computers}.

It is well-known (see, e.g.~\cite{Papadimitriou1982}) that the QUBO problem  can naturally represent a wide range of integer and combinatorial optimization problems. We now illustrate how a binary optimization problem with a quadratic objective function and linear constraints can be relaxed to QUBO problem. To this end, we adapt the framework from~\cite{Quintero2020}, which presents a QUBO relaxations for binary optimization problems with linear objective functions and linear constraints.

Consider linearly constrained quadratic binary  optimization problem:
\begin{align}\label{optimization_problem}
\min \left\{ x^T C x \mid Ax = b, \, x \in \{0, 1\}^n \right\},
\end{align}
where $C \in \mathbb{R}^{n \times n}$ is a symmetric matrix, $b \in \mathbb{R}^m$, and $A \in \mathbb{R}^{m \times n}$. We can associate to  this problem  a QUBO problem, obtained by moving the linear equality constraints $Ax = b$  into the objective using a quadratic penalty function $\frac{\mu}{2} \left \| Ax - b \right \|^2$, where $\mu > 0$ is a penalty parameter:
\begin{align}\label{reformulated_problem}
\min \left\{ x^T C x + \frac{\mu}{2} \left \| Ax - b \right \|^2 \mid x \in \{0, 1\}^n \right\}.
\end{align}

 The reformulated problem~\eqref{reformulated_problem} is a relaxation of the original problem ~\eqref{optimization_problem}, since every solution of  ~\eqref{optimization_problem} is, for obvious reasons, also a solution of ~\eqref{reformulated_problem} with the same objective value, while the opposite is not necessarily true. We say that the QUBO relaxation is an exact (also tight) relaxation if an optimal solution of ~\eqref{reformulated_problem} is also an optimal solution of the original problem ~\eqref{optimization_problem}.

A commonly used approach to ensure that the relaxation is exact involves setting the penalty term to exceed the largest possible objective value. Thus, as noted in~\cite{Quintero2020}, one possible choice for the formulation~\eqref{reformulated_problem} is to set the penalty term as:
\begin{align}\label{penalty_exact_method}
\frac{\mu}{2} > \frac{e^T C e}{\min \left\{ \left \| Ax - b \right \|^2 \mid Ax \neq b, ~x \in \{0, 1\}^n \right\}}.
\end{align}

While this condition is intuitively reasonable, its validity was not proven in~\cite{Quintero2020}. We therefore formalize and prove this result in the following proposition.\\

\begin{prop}\label{penalty_general_method}
Let $C \in \mathbb{R}^{n \times n}$ be symmetric and non-negative, and suppose $A \in \mathbb{R}^{m \times n}$ and $b \in \mathbb{R}^m$. Furthermore, let $\mu$ be as in~\eqref{penalty_exact_method}. Then any optimal solution $x^*$ of QUBO relaxation ~\eqref{reformulated_problem} is also an optimal solution of the original problem~\eqref{optimization_problem}.
\end{prop}

\begin{proof}
Let $x^*$ be an optimal solution of~\eqref{reformulated_problem} and suppose that $Ax^* \neq b$. Then, by the assumption on $\mu$, we have that
\begin{align*}
\frac{\mu}{2} \|A x^* - b \|^2 > e^T C e.
\end{align*}
Consequently, the objective value of the~\eqref{reformulated_problem} at $x^*$ satisfies
\begin{align*}
x^{*T} C x^* + \frac{\mu}{2} \|A x^* - b \|^2 > x^{*T} C x^* + e^T C e.
\end{align*}

Now let $\tilde{x}$ be any solution of~\eqref{reformulated_problem}, such that $A \tilde{x} = b$. Then the penalty term vanishes, and the objective function of ~\eqref{reformulated_problem} reduces to $\tilde{x} ^T C\tilde{x}$. Moreover, since $C$ is non-negative and $\tilde{x} \in \{0,1\}^n$, it holds that $\tilde{x}^T C \tilde{x} \leq e^T C e$. Combining this with we the inequality above, we obtain
\begin{align*}
x^{*T} C x^* + \frac{\mu}{2} \|A x^* - b \|^2 > \tilde{x}^T C \tilde{x},
\end{align*}
which contradicts the optimality of $x^*$ for ~\eqref{reformulated_problem}.

Hence, $x^*$ must be feasible for the constrained problem~\eqref{optimization_problem}. For every feasible solution, the penalty term is zero, so the  objective function of ~\eqref{reformulated_problem} 
reduces to $x^T C x$. Therefore,
\begin{align*}
x^* \in \arg\min \left\{ x^T C x \mid Ax = b,\, x \in \{0,1\}^n \right\},
\end{align*}
which means that $x^*$ is optimum for ~\eqref{optimization_problem}.
\end{proof}

Choosing the quadratic penalty parameter $\mu$  when working with quantum annealers is not a trivial task. A detailed discussion of this issue can be found in~\cite{Ayodele2022}. For instance, the authors of \cite{Verma2022} report that if $\mu$ is too large, it can dominate the original objective function, making it challenging to obtain feasible solutions.

On the hardware such as D-Wave Quantum Processing Unit (QPU), all coefficients are normalized to meet device restrictions. A large $\mu$ results in some coefficients from the original objective function becoming very small in absolute value, which reduces the accuracy of quantum annealing~\cite{Djidjev2023}. In contrast, if $\mu$ is too small, infeasible solutions may dominate, leading to incorrect results. In practice, there is not a unique candidate for $\mu$,  and a range of values often provides effective results~\cite{Glover2019}.

The augmented Lagrangian approach is a well-known approach from the mathematical optimization area \cite{Bertsekas} which combines the quadratic penalty method with Lagrangian multipliers. Specifically, this method adds a linear Lagrangian term to the quadratic penalty:
\begin{align*}
\min \left\{ x^T C x + \lambda^T (Ax - b) + \frac{\mu}{2} \left \| Ax - b \right \|^2 \mid x \in \{0, 1\}^n \right\},
\end{align*}
and employs an iterative procedure to determine optimal or near-optimal values for $\lambda$ and $\mu$. 

When combined with quantum annealing, the augmented Lagrangian method aims to address challenges related to the reduced precision in quantum annealers when handling large-scale problems. However, as noted in~\cite{Djidjev2023_2}, this method was initially developed for continuous variables. In that setting, and under mild conditions, it converges to a global optimum regardless of the initial values of the used parameters; see, for instance,~\cite{Bertsekas, Fernandez2012}. For discrete variables, however, there are no theoretical results that guarantee optimality.%Despite this, experimental results in~\cite{Djidjev2023} demonstrated that in the context of quantum annealing, this approach outperforms the standard penalty method when applied to the set cover problem.

A third approach is to omit the quadratic penalty term entirely and instead consider the Lagrangian relaxation of the problem~\eqref{optimization_problem}, defined as:
\begin{align*}
\min \left\{ x^T C x + \lambda^T (Ax - b)  \mid x \in \{0, 1\}^n \right\}.
\end{align*}
As in the augmented Lagrangian method, the Lagrangian multiplier $\lambda$ is determined in an iterative procedure that yields optimal or near-optimal values for $\lambda$. A detailed study of Lagrangian relaxation for continuous optimization problems can be found in~\cite{Ahuja1993}, while its application to integer programming is explored in~\cite{Fisher2004}. Additionally, a method for solving the Lagrangian dual of a constrained binary quadratic programming problem using a quantum annealer is presented in~\cite{Ronah}.

\section{The sparsest $k$-subgraph problem}\label{definitions}

The SkS problem is a combinatorial optimization problem where the objective is to find a subgraph on $k$ vertices of a given graph that contains the minimum number of edges. For a given graph $G = (V, E)$ with vertex set $V = \{1, \ldots, n\}$ and a positive integer $k$, we can model subgraphs of $G$ by binary incidence vectors $x\in\{0,1\}^n$ with $x_i=1$ if and only if the vertex $i$ is included in the subgraph. 

The condition that the size of the subgraph defined by a binary vector $x$ must be equal to $k$ can be modeled by the constraint $\sum_{i \in V} x_i=k$. The number of edges in the subgraph defined by $x$ can be expressed as $\sum_{\{i,j\} \in E} x_i x_j$. Therefore, the SkS problem can be formulated as the following quadratic binary optimization problem with one linear constraint:
\begin{align}\label{SkS_problem_standard_formulation_0}
     m_k:=\min \left\{ \sum_{\{i,j\} \in E} x_i x_j \mid e^T x = k, ~x \in \{0, 1\}^n \right\}.
\end{align}

Now, let $A$ denote the adjacency matrix of $G$. Then, $\sum_{\{i,j\} \in E} x_i x_j = \frac{1}{2}x^T A x$, so the SkS problem can be equivalently written as:
\begin{align}\label{SkS_problem_standard_formulation}
\tag{\mbox{SkS}}
m_k= \min \left\{ \frac{1}{2}x^T A x \mid e^T x = k, 
~x \in \{0, 1\}^n \right\}.
\end{align}

Any optimum solution of ~\eqref{SkS_problem_standard_formulation} is called the sparsest $k$-subgraph of the given graph.

The SkS problem is closely related to the densest $k$-subgraph (DkS) problem, where the objective is to find a subgraph on $k$ vertices that contains the maximum number of edges. 
The relationship between the SkS and DkS problems is rather straightforward: a sparsest $k$-subgraph in $G$ is a densest $k$-subgraph in $\overline{G}$, and vice versa. This equivalence arises because minimizing the number of subgraph edges across all subgraphs of $G$ inherently maximizes the number of subgraph edges across all subgraphs in $\overline{G}$, where the edges correspond to those absent in $G$. 

Another problem related to the SkS problem is the maximum stable set problem. Given a graph $G$ that follows the general assumptions underlying this paper, a subset of vertices $S \subseteq V$ is called a stable set if the subgraph induced by $S$ has no edge. A maximum stable set is a stable set of the largest cardinality. The cardinality of the maximum stable set in $G$ is called the stability number of $G$ and is denoted by $\alpha(G)$. 

The relationship between the SkS problem and the maximum stable set problem can be summarized as follows. For a given graph $G$ on $n$ vertices and a positive integer $k \leq n$, let $x^*$ be an optimal solution of the problem~\eqref{SkS_problem_standard_formulation}. If $\frac{1}{2}x^{*T} A x^* = 0$, then $x^*$ is the incidence vector of a stable set in $G$. Therefore, finding the largest $k$ for which the objective function in~\eqref{SkS_problem_standard_formulation} equals zero yields the stability number $\alpha(G)$ of the given graph $G$. Note that this relationship was also observed in~\cite{PuchRen2023}, where semidefinite programming relaxations of the problem~\eqref{SkS_problem_standard_formulation}, as well as its connection to semidefinite programming relaxations for the maximum stable set problem, were investigated.

We assume that solving QUBO reformulations is more convenient than solving the original problem, in part because such reformulations can also be addressed using general exact solvers for QUBO problems, quantum annealers, and a wide range of meta-heuristic algorithms. Therefore, the central question we investigate for each QUBO reformulation is whether an optimal solution of a QUBO reformulation is also an optimal solution of the original problem ~\eqref{SkS_problem_standard_formulation}.

In the rest of this section,  we present several QUBO formulations and relaxations for the ~\eqref{SkS_problem_standard_formulation} problem and discuss their properties and effectiveness. 
We begin with a formulation based on the  quadratic penalty method, and we explore different strategies for selecting the penalty parameter.
Next, we discuss Lagrangian relaxation of the ~\eqref{SkS_problem_standard_formulation} problem. Instead of incorporating a quadratic penalty term, this approach employs a linear term containing the cardinality constraint. 
Finally, we combine Lagrangian relaxation with the quadratic penalty method and consider the augmented Lagrangian approach for ~\eqref{SkS_problem_standard_formulation}. We show that for any graph, it is possible to choose the Lagrangian multiplier and quadratic penalty term in a way that guarantees an optimal solution. 

\subsection{Quadratic penalty relaxation  of \ref{SkS_problem_standard_formulation}}\label{section_SkS_QUBO_exact_penalty}

The first QUBO relaxation of ~\eqref{SkS_problem_standard_formulation} is based on the standard quadratic penalty method presented in Section~\ref{QUBO_theory}, where we penalize the cardinality constraint $e^Tx = k$ by moving it as a quadratic term into the objective function, resulting in the following QUBO formulation
\begin{align}\label{SkS_QUBO_formulation_one}
\tag{\mbox{SkS-Q}}
     \min \left\{\frac{1}{2}x^TAx + \frac{\mu}{2}(e^Tx - k)^2 \mid x \in \{0, 1\}^n \right\},
\end{align}
where $\mu \in \mathbb{R}^+$ is a  penalization parameter.

Note that the term $\frac{\mu}{2}(e^Tx - k)^2$ in~\eqref{SkS_QUBO_formulation_one} penalizes all cases where $e^Tx \neq k$, supporting that exactly $k$ vertices are selected in an optimal solution of ~\eqref{SkS_QUBO_formulation_one}. 
We first prove a general lemma, connecting optimal solutions of ~\eqref{SkS_QUBO_formulation_one} with optimal solutions of ~\eqref{SkS_problem_standard_formulation}.\\

\begin{lem}
\label{lem:optimality_test_QUBO}
Let $G$ be a graph on $n$ vertices and $k \leq n$. Furthermore, let $x^* \in \{0,1\}^n$ be an optimal solution of~\eqref{SkS_QUBO_formulation_one} for some $\mu > 0$. If $e^T x^* = k$, then $x^*$ is also an optimal solution of~\eqref{SkS_problem_standard_formulation}.
\end{lem}

\begin{proof}
Since $x^*$ is optimal for~\eqref{SkS_QUBO_formulation_one}, we have
\begin{align*}
\frac{1}{2} x^{*T} A x^* + \frac{\mu}{2} (e^T x^* - k)^2 \leq \frac{1}{2} x^T A x + \frac{\mu}{2} (e^T x - k)^2
\end{align*}
for all $x \in \{0,1\}^n$. Suppose that $x^*$ satisfies $e^T x^* = k$. Then the penalty term vanishes, and therefore
\begin{align*}
\frac{1}{2} x^{*T} A x^* \leq \frac{1}{2} x^T A x + \frac{\mu}{2} (e^T x - k)^2
\end{align*}
for all $x \in \{0,1\}^n$. In particular, this inequality holds for all $x$ satisfying $e^T x = k$, where the penalty term also vanishes. That is,
\begin{align*}
\frac{1}{2} x^{*T} A x^* \leq \frac{1}{2} x^T A x 
\end{align*}
for all $x \in \{0,1\}^n$ with $e^T x = k$.

Furthermore, given that $e^T x^* = k$, $x^*$ is feasible for~\eqref{SkS_problem_standard_formulation}. Now assume that $x^*$ is not optimal for~\eqref{SkS_problem_standard_formulation}. Then there exists a solution $\hat{x} \in \{0,1\}^n$ with $e^T \hat{x} = k$ such that
\begin{align*}
\frac{1}{2} \hat{x}^T A \hat{x} < \frac{1}{2} x^{*T} A x^*.
\end{align*}
But this contradicts the inequality above, which states that $x^*$ minimizes $\frac{1}{2} x^T A x$ among all binary vectors of cardinality $k$. Therefore, $x^*$ is an optimal solution of~\eqref{SkS_problem_standard_formulation}.
\end{proof}

The lemma above thus guarantees that optimal solutions of~\eqref{SkS_QUBO_formulation_one}, which are feasible with respect to the original problem~\eqref{SkS_problem_standard_formulation}, are also optimal solutions to the original problem. This leads to the following question: how can we ensure, through an appropriate choice of the penalty parameter $\mu$, that the optimal solutions of~\eqref{SkS_QUBO_formulation_one} are indeed feasible---and therefore optimal---for the original problem~\eqref{SkS_QUBO_formulation_one}?

Given that the matrix $A$ in the formulation~\eqref{SkS_QUBO_formulation_one} is non-negative, we can use inequality~\eqref{penalty_exact_method} and get our first lower bound for $\frac{\mu}{2}$: 
\begin{align*}
    \frac{\mu}{2} > \frac{\frac{1}{2}e^T A e}{\min \left\{ \left \| e^Tx - k \right \|^2 \mid e^Tx \neq k, ~x \in \{0, 1\}^n \right\}} = \frac{m}{1} = m,
\end{align*}
where $m$ is the number of edges in the underlying graph $G$.

In Section \ref{QUBO_theory}, we expose challenges related to too large or too small values of $\mu$. The lower bound $m$ may already be too large if the graph is dense; therefore, in the following, we propose an alternative approach that results in a more appropriate penalty parameter value.

\begin{lem}\label{penalty_feasible_solution} 
Let $G$ be a graph on $n$ vertices and $k \leq n$. Let $\tilde{x} \in \{0, 1\}^n$ such that $e^T\tilde{x} = k$, and let $\frac{\mu}{2} > \frac{1}{2} \tilde{x}^TA\tilde{x}$. Then,
\begin{itemize}
    \item [(i)]
\begin{align*}
\frac{1}{2} \tilde{x}^TA\tilde{x} < \frac{1}{2} x^T A x + \frac{\mu}{2}(e^Tx - k)^2
\end{align*}
for all $x \in \{0, 1\}^n$ with $e^Tx \neq k$.
\item[(ii)] Any  optimal solution of~\eqref{SkS_QUBO_formulation_one} is an optimal solution of ~\eqref{SkS_problem_standard_formulation}.
\end{itemize}
\end{lem}

\begin{proof}
To prove {\it (i)}, suppose  $x \in \{0, 1\}^n$ with $e^Tx \neq k$. Then, $(e^Tx - k)^2 \geq 1$, so
\begin{align*}
\frac{1}{2} x^T A x + \frac{\mu}{2}(e^Tx - k)^2 \geq \frac{1}{2} x^T A x + \frac{\mu}{2} > \frac{1}{2} x^T A x + \frac{1}{2} \tilde{x}^TA\tilde{x} \geq \frac{1}{2} \tilde{x}^TA\tilde{x}.
\end{align*}

For {\it (ii)}, let $x$ be an optimal solution of ~\eqref{SkS_QUBO_formulation_one}.
If  $e^Tx\neq k$, then {(\it i)} implies that 
\begin{align*}
\frac{1}{2} x^T A x + \frac{\mu}{2}(e^Tx - k)^2 > \frac{1}{2} \tilde{x}^TA\tilde{x}+
\frac{\mu}{2}(e^T \tilde x - k)^2,
\end{align*}
hence, $x$ can not be an optimum for 
\eqref{SkS_QUBO_formulation_one}.
Therefore, $x$ is feasible for ~\eqref{SkS_problem_standard_formulation}. If it is not optimum for  ~\eqref{SkS_problem_standard_formulation}, then there exists $\hat x$ such that $e^T\hat x=k$ and 
$$\frac{1}{2} x^T A x > \frac{1}{2} \hat x^T A \hat x.$$ 
But in this case, we would have 
\begin{align*}
\frac{1}{2} x^T A x + \frac{\mu}{2}(e^Tx - k)^2 > 
\frac{1}{2} \hat x^T A \hat x + \frac{\mu}{2}(e^T\hat x - k)^2,
\end{align*}
which implies that  $x$ is not an optimum for   ~\eqref{SkS_QUBO_formulation_one}, a contradiction.
\end{proof}

If we manage to find a feasible solution for ~\eqref{SkS_problem_standard_formulation} with some efficient heuristics, where the number of edges is rather small, then we know from Lemma~\ref{penalty_feasible_solution} that we can choose a small penalty parameter. However, the number of edges in feasible solutions can still be large, which is why we continue to investigate whether it is possible to choose even a smaller penalty parameter.

By definition, we have  $m_k \leq m_{k+1}$ for all $k \in \{1, \ldots, n-1\}$. Furthermore, note that a sparsest subgraph on $k+1$ vertices can have at most $k$ edges more than a sparsest subgraph on $k$ vertices, while a sparsest subgraph on $k+2$ vertices has at most $k + k + 1 = 2k + 1$ more edges than a sparsest subgraph on $k$ vertices. Thus, if we deal with a graph on $n$ vertices, then a sparsest subgraph on $k+i$ vertices for some $i \in \{1, \ldots, n-k\}$, can have at most 
\begin{align*}
k + k + 1 + \ldots + k + i - 1 = ki + \frac{i(i-1)}{2}
\end{align*}
more edges than a sparsest $k$-subgraph in $G$, i.e., we have
\begin{equation}\label{eq:m_k_1}
m_{k+i}\le m_k+ki + \frac{i(i-1)}{2}.
\end{equation}

Similarly, a sparsest subgraph on $k$ vertices can have at most 
\begin{equation*} 
k - 1 + k - 2 + \ldots + k - j = kj - \frac{j(j+1)}{2}
\end{equation*}
more edges than a sparsest subgraph on $k-j$ vertices for any $j \in \{1, \ldots, k - 1\}$, i.e., the following inequality holds:
\begin{equation}\label{eq:m_k_2}
m_{k}\le m_{k-j}+kj - \frac{j(j+1)}{2}.
\end{equation}
This is a basis for the following lemma.\\

\begin{lem}\label{lemma_k_minus_1}
Let $G$ be a graph on $n$ vertices and $k \leq n$. Let $x^* \in \{0, 1\}^n$ be an optimal solution of ~\eqref{SkS_problem_standard_formulation} and let $\frac{\mu}{2} > k - 1$. Then,
\begin{itemize}
    \item [(i)]
\begin{align}\label{relation_k_minus_1_exact}
\frac{1}{2} x^{*T} A x^* < \frac{1}{2} x^T A x + \frac{\mu}{2}(e^Tx - k)^2
\end{align}
for all $x \in \{0, 1\}^n$ with $e^Tx \neq k$. 
    \item [(ii)] Any  optimal solution of~\eqref{SkS_QUBO_formulation_one}  is an optimal solution of ~\eqref{SkS_problem_standard_formulation}.
\end{itemize}
\end{lem} 

\begin{proof}
We first prove {\it (i)}.
Let $x \in \{0, 1\}^n$ such that $e^T x = k + i$ for some $i \in \{1, \ldots, n-k\}$. Then, 
\begin{align}\label{in_1_lemma_k_minus_1}
\frac{1}{2} x^T A x + \frac{\mu}{2}(e^Tx - k)^2 \geq m_{k+i} + \frac{\mu}{2} i^2 > m_{k+i} +  (k-1) i^2 \geq m_k = \frac{1}{2} x^{*T} A x^*.
\end{align}
 Since~\eqref{in_1_lemma_k_minus_1} holds for all $i \in \{1, \ldots, n-k\}$, we obtain that~\eqref{relation_k_minus_1_exact} holds for all $x \in \{0, 1\}^n$ with $e^Tx > k$. 

Now let $x \in \{0, 1\}^n$ such that $e^T x = k - j$ for some $j \in \{1, \ldots, k - 1\}$. Then,
\begin{align}\label{in_2_lemma_k_minus_1}
\frac{1}{2} x^T A x + \frac{\mu}{2}(e^Tx - k)^2 > m_{k - j} + (k-1) j^2 \geq m_{k-j} + kj - \frac{j(j+1)}{2} \geq m_k = x^{*T} A x^*.
\end{align}
Note that the second inequality in the above chain of inequalities is not entirely straightforward; however, it is essentially a standard calculus task, so we omit its proof.

Since~\eqref{in_2_lemma_k_minus_1} holds for all $j \in \{1, \ldots, k-1\}$, this implies that~\eqref{relation_k_minus_1_exact} is satisfied also for all $x \in \{0, 1\}^n$ with $e^Tx < k$, which completes the proof of {\it (i)}.

To prove {\it (ii)}, let us assume that $\hat x$ is optimal for ~\eqref{SkS_QUBO_formulation_one}. If $e^T\hat x\neq k$, then {\it (i)} implies
$$\frac{1}{2} \hat x^T A \hat x + \frac{\mu}{2}(e^T\hat x - k)^2 > \frac{1}{2} x^{*T} A x^* +
\frac{\mu}{2}(e^{T} x^*- k)^2,$$
which contradicts the assumption of optimality of $\hat x$ for ~\eqref{SkS_QUBO_formulation_one}.
Therefore, $\hat x$ is feasible for ~\eqref{SkS_problem_standard_formulation}.
If it is not an optimal solution for ~\eqref{SkS_problem_standard_formulation}, then 
we have
$$\frac{1}{2} \hat x^T A \hat x+\frac{\mu}{2}(e^T\hat x - k)^2=\frac{1}{2} \hat x^T A \hat x > \frac{1}{2} x^{*T} A x^*=\frac{1}{2} x^{*T} A x^* +
\frac{\mu}{2}(e^{T} x^*- k)^2,$$
which again contradicts the optimality  of $\hat x$ for ~\eqref{SkS_QUBO_formulation_one}.
\end{proof}

Thus, according to Lemmas~\ref{penalty_feasible_solution} and~\ref{lemma_k_minus_1}, we have two possibilities to choose the value of the penalty parameter $\mu$. Moreover, we can combine these results as well. Let $\tilde{m}_k$ denote the number of edges in a feasible solution for ~\eqref{SkS_problem_standard_formulation}. Then, we can choose the parameter $\mu$ as
\begin{align}\label{choose_quadratic_penalty}
\frac{\mu}{2} > \min \left\{ \tilde{m}_k, k - 1 \right\}.
\end{align}

Next, we demonstrate that in some cases, even lower penalty parameters can be chosen. For this purpose, for a given $k \in \{2, \ldots, n\}$ we define sequence $\{\diff_\ell\mid 2\le \ell \le n\}$ as follows 
\begin{align*}
\diff_\ell : = m_\ell - m_{\ell-1},
\end{align*}
and for $\ell=1$, we set $\diff_1 = m_1=0$. Furthermore, we say that the sequence $\{\diff_\ell\}$ is monotonically increasing in $G$ 
if $\diff_{\ell} \leq \diff_{\ell+1}$ for all $\ell \in \{1, \ldots, n-1\}$ and therewith
\begin{align*}
0 \leq m_2 - m_{1} \leq m_3 - m_{2} \leq \ldots \leq m_{n-1} - m_{n} \leq m_n - m_{n-1}.
\end{align*}

Note that graphs where $\{\diff_\ell\}$ is monotonically increasing exist. Some trivial examples of such graphs are empty graphs (graphs with no edges), for which we have 
$\diff_1 = \diff_2 = \diff_3 = \ldots =  \diff_n = 0$, and complete graphs, for which we have 
$\diff_1 = 0,  \diff_2 = 1,  \ldots, \diff_k=k-1,\ldots,   \diff_n = n-1.$
\\

\begin{lem}\label{quadratic_SkS_lower_penalty}
Let $G$ be a graph on $n$ vertices and $k \leq n$. Suppose that the sequence $\{\diff_{\ell}\}$ is monotonically increasing in $G$. Let 
 $\frac{\mu}{2} > \diff_k$ and $x^* \in \{0, 1\}^n$ be an optimal solution for ~\eqref{SkS_problem_standard_formulation}. Then,
\begin{itemize}
\item [(i)]
    \begin{align}\label{SkS_quadratic_3}
\frac{1}{2} x^{*T} A x^* < \frac{1}{2} x^T A x + \frac{\mu}{2}(e^Tx - k)^2
\end{align}
for all $x \in \{0, 1\}^n$ with $e^Tx \neq k$. 

   \item [(ii)]Any  optimal solution of~\eqref{SkS_QUBO_formulation_one}  is an optimal solution of ~\eqref{SkS_problem_standard_formulation}.
\end{itemize}
\end{lem} 

\begin{proof}
To prove {\it (i)}, let $x \in \{0, 1\}^n$ such that $e^Tx = k + i$ for some $i \in \{1, \ldots, n-k\}$. Then,
\begin{align}\label{relation_1_smaller_than_k}
\frac{1}{2} x^T A x + \frac{\mu}{2}(e^Tx - k)^2 \geq m_{k+i} + \frac{\mu}{2} i^2 \geq m_{k+i} + \frac{\mu}{2}  > m_{k+i} + m_k - m_{k-1} \geq m_k,
\end{align}
since $m_{k+i} \geq m_{k-1}$. Furthermore, note that $m_k = \frac{1}{2} x^{*T} A x^*$. Given that~\eqref{relation_1_smaller_than_k} holds for all $i \in \{1, \ldots, n-k\}$, we obtain that~\eqref{SkS_quadratic_3} holds for all $x \in \{0, 1\}^n$ with $e^Tx > k$. 

Now let $x \in \{0, 1\}^n$ such that $e^T x = k - j$ for some $j \in \{1, \ldots, k - 1\}$. First, we recall that, according to assumption, $\{\diff_\ell\}$ is monotonically increasing in $G$, so by assumption on $\frac{\mu}{2}$, we have that
\begin{align*}
\frac{\mu}{2} > m_{k} - m_{k-1} \geq m_{k-1} - m_{k-2} \geq \ldots \geq m_{k-j+1} - m_{k-j},
\end{align*}
and therefore
\begin{align*}
\frac{\mu}{2} j > (m_k - m_{k-1}) + (m_{k-1} - m_{k-2}) + \ldots + (m_{k-j+1} - m_{k-j}) = m_k - m_{k - j}.
\end{align*}

Hence,
\begin{align}\label{relation_2_smaller_than_k}
\frac{1}{2} x^T A x + \frac{\mu}{2}(e^Tx - k)^2 \geq m_{k-j} + \frac{\mu}{2} j^2 \geq m_{k-j} + \frac{\mu}{2} j > m_{k-j} + m_k - m_{k-j} \geq m_k.
\end{align}
Since~\eqref{relation_2_smaller_than_k} is valid for any $j \in \{1, \ldots, k-1\}$ inequality ~\eqref{SkS_quadratic_3} also holds for all $x \in \{0, 1\}^n$ with $e^Tx < k$, which finishes the proof of {\it (i)}.

The proof of {\it (ii)} essentially follows the proofs of Lemmas \ref{penalty_feasible_solution} and \ref{lemma_k_minus_1}.
Suppose $\hat x$ is an optimum of ~\eqref{SkS_QUBO_formulation_one}. If $\hat x$ is not feasible for ~\eqref{SkS_problem_standard_formulation}, it can not be optimal for 
\eqref{SkS_QUBO_formulation_one}, so it holds $e^T\hat x=k$.
If $\hat x$ is not an optimum of ~\eqref{SkS_problem_standard_formulation} then again it can not be optimal for  ~\eqref{SkS_QUBO_formulation_one}, so $\hat x$ must also be an optimum for 
\eqref{SkS_problem_standard_formulation}.
\end{proof}

Hence, from Lemma~\ref{quadratic_SkS_lower_penalty} we know that in certain cases, we can choose even a lower penalty parameter than the one proposed in~\eqref{choose_quadratic_penalty}. However, we cannot determine in advance whether $\{\diff_\ell\}$ will satisfy the required property for a given graph. Additionally, we do not know the value of the sequence $\{\diff_\ell\}$. Nevertheless, we can estimate its value by computing feasible solutions with $k$ and $k-1$ vertices. Then, we can attempt to solve the problem using this approximated penalty parameter. If the obtained optimal solution has the desired cardinality, we have an optimal solution to the original problem. Otherwise, we revert to the standard approach and choose the penalty term as in~\eqref{choose_quadratic_penalty} to ensure exactness.

\subsection{Lagrangian  relaxation of \ref{SkS_problem_standard_formulation}}
\label{section_LR_SkS}

In this section, we consider the Lagrangian relaxation of~\eqref{SkS_problem_standard_formulation}, which means that we move the linear constraint $e^Tx=k$, multiplied with a (Lagrangian dual) parameter $\lambda$, to the objective function, see, for instance,~\cite{Ahuja1993}.
For the sake of simplicity, we move to the objective function $\lambda(k-e^Tx)$ instead of 
$\lambda(e^Tx-k)$. These are clearly equivalent because the constraint is an equality, and therefore $\lambda$ is an unconstrained dual variable.

The Lagrangian relaxation of ~\eqref{SkS_problem_standard_formulation} is therefore: 
\begin{align}\label{SkS_QUBO_Lagrangian}
\tag{\mbox{SkS-L}}
     \min \left\{\frac{1}{2}x^TAx + \lambda(k-e^Tx) \mid x \in \{0, 1\}^n \right\}.
\end{align}

Note that the classical dual theory implies that for any $\lambda \in \mathbb{R}$, the optimum value of the relaxation~\eqref{SkS_QUBO_Lagrangian} is a lower bound on the optimum value of~\eqref{SkS_problem_standard_formulation}; for details, see~\cite{Ahuja1993}.
Since the term $\lambda k$ is a constant, we can leave it out of the objective function and solve only the following problem:
\begin{align} 
\tag{\mbox{SkS-L'}}
  \min \left\{\frac{1}{2}x^TAx - \lambda e^Tx \mid x \in \{0, 1\}^n \right\}.
  \label{SkS_QUBO_Lagrangian_1} 
\end{align} 
We get the optimum value of ~\eqref{SkS_QUBO_Lagrangian} from the optimum value of 
\eqref{SkS_QUBO_Lagrangian_1} by adding $\lambda k$. 

%\subsubsection{Some results about $\lambda$} \label{subsubsec: existence_lambda}
Suppose now that for some $\lambda \in  \mathbb{R}$, the optimal solution of~\eqref{SkS_QUBO_Lagrangian} satisfies the desired cardinality constraint $k$. Then, as established in the following lemma—analogous to Lemma \ref{lem:optimality_test_QUBO}—this solution is also optimal for the original formulation ~\eqref{SkS_problem_standard_formulation}.\\

\begin{lem}\label{lem:optimality_test_LR}
Let $G$ be a graph on $n$ vertices and $k\le n$. Furthermore, let $x^*$ be an optimal solution of the relaxation~\eqref{SkS_QUBO_Lagrangian_1} for some $\lambda \in \mathbb{R}$. If $e^T x^* = k$, then $x^*$ is also an optimal solution of the original problem~\eqref{SkS_problem_standard_formulation}.
\end{lem}

\begin{proof}
We proceed in the same manner as in the proof of Lemma~\ref{lem:optimality_test_QUBO}. First, given that $x^*$ is optimal for~\eqref{SkS_QUBO_Lagrangian_1}, we note that
\begin{align*}
\frac{1}{2}x^{*T} A x^* - \lambda e^T x^* \leq \frac{1}{2} x^T A x - \lambda  e^T x
\end{align*}
for all $x \in \{0,1\}^n$. The assumption that  $e^T x^* = k$ implies 
\begin{align*}
\frac{1}{2} x^{*T} A x^* -\lambda k\leq \frac{1}{2} x^T A x - \lambda  e^T x
\end{align*}
for all $x \in \{0,1\}^n$, and, in particular,
\begin{align*}
\frac{1}{2}x^{*T} A x^* \leq \frac{1}{2} x^T A x 
\end{align*}
for all $x \in \{0,1\}^n$ with $e^T x = k$. Hence $x^*$ is optimal also for~\eqref{SkS_problem_standard_formulation}.
\end{proof}

Similarly to the quadratic penalty approach, we will provide some theoretical results about the values of $\lambda$  which ensure that the optimal solutions of ~\eqref{SkS_QUBO_Lagrangian_1} are optimal also for the original problem.

We start with some simple observations. Recall that the multiplier $\lambda$ is not constrained and can take both positive and negative values. However, if $\lambda < 0$, then the problem~\eqref{SkS_QUBO_Lagrangian_1} becomes
\begin{align}
\min \left\{\frac{1}{2}x^TAx + \tilde{\lambda} e^Tx \mid x \in \{0, 1\}^n \right\}\label{negative_lambda}
\end{align}
for $\tilde{\lambda} = - \lambda > 0$. Clearly, the minimum of~\eqref{negative_lambda} equals zero and is attained when $x$ is the zero vector.

If $\lambda = 0$, then the problem~\eqref{SkS_QUBO_Lagrangian_1} is reduced to minimizing the number of edges $\frac{1}{2}x^TAx$. Since the minimum value of this objective function is zero, any incidence vector of a stable set in the graph, as well as the zero vector, is a feasible solution.

If $\lambda = 1$, then solving the problem~\eqref{SkS_QUBO_Lagrangian_1} becomes equivalent to solving the QUBO formulation for the maximum stable set problem with the penalty parameter $\frac{1}{2}$. In particular, for a given $G$, the optimal value of~\eqref{SkS_QUBO_Lagrangian_1} in this case corresponds to $-\alpha(G)$; for details, see~\cite[Lemma 2]{Krpan2024}. Moreover, for any $\lambda \in (0, 1)$, any optimal solution of~\eqref{SkS_QUBO_Lagrangian_1} is the incidence vector of a maximum stable set, as the next statement shows.\\

\begin{lem}\label{stable_set_lagrangian}
Let $G$ be a graph on $n$ vertices. If $\lambda \in (0,1)$, then any optimum solution of~\eqref{SkS_QUBO_Lagrangian_1}  is the incidence vector of a maximum stable set in $G$.
\end{lem}

\begin{proof}
Let $x^*$ be the incidence vector of a maximum stable set in $G$. Then, clearly, $e^T x^* = \alpha(G)$ and $\frac{1}{2} x^{*T} A x^* = 0$. Now, assume for contradiction that $\lambda \in (0,1)$ but $x^*$ is not an optimal solution of the formulation~\eqref{SkS_QUBO_Lagrangian_1}. Let $\hat{x} \in \{0,1\}^n$ be  an optimal solution of~\eqref{SkS_QUBO_Lagrangian_1} that is not the incidence vector of a maximum stable set in $G$. 

If $\hat{x}$ is the incidence vector of a stable set in $G$, then $e^T \hat{x} < \alpha(G)$ and $\frac{1}{2} \hat{x}^{T} A \hat{x} = 0$. However, since $\lambda > 0$, we obtain
\begin{align*}
\frac{1}{2} \hat{x}^T A \hat{x} - \lambda e^T \hat{x} = - \lambda e^T \hat{x} > - \lambda \alpha(G) = \frac{1}{2} x^{*T} A x^* - \lambda e^T x^*,
\end{align*}
contradicting the optimality of $\hat{x}$.

If $\hat{x}$ is not the incidence vector of a stable set in $G$, then $\frac{1}{2} \hat{x}^T A \hat{x} > 0$ holds. We must consider two cases:  (i) if $e^T \hat{x} \leq \alpha(G)$, then
\begin{align*}
\frac{1}{2} \hat{x}^T A \hat{x} - \lambda e^T \hat{x} >  -\lambda e^T \hat{x} \geq - \lambda \alpha(G) = \frac{1}{2} x^{*T} A x^* - \lambda e^T x^*,
\end{align*}
again contradicting the optimality of $\hat{x}$.

(ii) If $e^T \hat{x} > \alpha(G)$, i.e., $e^T \hat{x} = \alpha(G) + j$ for some $j \geq 1$, then
we have  $\frac{1}{2}\hat{x}^T A \hat{x} \geq j$.
This follows from the observation that if the graph induced by $\hat{x}$, which has $\alpha(G) + j$ vertices, has less than $j$ edges, then we can remove from this induced graph one end vertex for  each edge  and obtain this way a stable set with more than $\alpha(G)$ vertices, which is a contradiction.

Therefore, 
\begin{align*}
\frac{1}{2} \hat{x}^T A \hat{x} - \lambda e^T \hat{x} 
& \geq j - \lambda (\alpha(G) + j) = - \lambda \alpha(G) + (1 - \lambda) j\\
& > - \lambda \alpha(G) = \frac{1}{2} x^{*T} A x^* - \lambda e^T x^*,
\end{align*}
since $\lambda \in (0, 1)$ and $j \geq 1$. But this again contradicts the assumption on optimality of $\hat{x}$. Hence, an optimal solution of~\eqref{SkS_QUBO_Lagrangian_1} must be the incidence vector of a maximum stable set in $G$.
\end{proof}

Therefore, if we solve ~\eqref{SkS_QUBO_Lagrangian_1} with $\lambda \in (0, 1)$ and $k < \alpha(G)$, then we can extract from computed optimal solution (which is a maximum stable set) an optimal solution of~\eqref{SkS_problem_standard_formulation} with cardinality $k$ by simply taking any subset of $k$ vertices from the optimum solution.

Within the rest of this subsection, we establish several partial results which demonstrate relations between the values of $\lambda$ and the solutions of \eqref{SkS_QUBO_Lagrangian_1}.

\begin{lem}\label{lambda_greater_k}
Let $G$ be a graph on $n$ vertices, and let $k$ be a positive integer such that $k \leq n - 1$. If $\lambda > k$, then no optimal solution of~\eqref{SkS_QUBO_Lagrangian_1} corresponds to an optimal solution of ~\eqref{SkS_problem_standard_formulation}.
\end{lem}

\begin{proof}
Assume for contradiction that $\lambda > k$ and that there exists an incidence vector $x^*$ which is an optimal solution to~\eqref{SkS_QUBO_Lagrangian_1} and to ~\eqref{SkS_problem_standard_formulation}. 
Let $m^*$ denote the number of edges in the subgraph induced by $x^*$. By optimality, we have: 
\begin{align}\label{otptimal_sol_k}
\frac{1}{2} x^{*T} A x^* - \lambda e^T x^* = m^* - \lambda k \leq \frac{1}{2} x^T A x - \lambda e^T x
\end{align}
for all $x \in \{0, 1\}^n$.

Now, consider $\tilde{x}$, the incidence vector of a sparsest $(k+1)$-subgraph in $G$. Note that such a subgraph exists because $k \leq n - 1$. Furthermore, let $\tilde{m}$ denote the number of edges in the subgraph induced by $\tilde{x}$. Since a sparsest $(k+1)$-subgraph in $G$ can have at most $k$ edges more than a sparsest $k$-subgraph in $G$, we have that $\tilde{m} - m^* = a$ for some $a \in \{0, \ldots, k\}$.

Substituting $\tilde{x}$ in~\eqref{otptimal_sol_k}, we obtain:
\begin{align*}
m^* - \lambda k & \leq \frac{1}{2} \tilde{x}^T A \tilde{x} - \lambda e^T \tilde{x} = \tilde{m} - \lambda (k+1) = m^* + a - \lambda k - \lambda,
\end{align*}
and therefore $\lambda \leq a$. But $a \in \{0, \dots, k\}$, implying that $\lambda \leq k$.

This contradicts the assumption that $\lambda > k$. Therefore, an optimal solution to~\eqref{SkS_QUBO_Lagrangian_1} cannot be optimal to~\eqref{SkS_problem_standard_formulation}, when $\lambda > k$. 
\end{proof}

\begin{prop}\label{roman:obs}
Let $x^*$ be a minimizer of~\eqref{SkS_QUBO_Lagrangian_1} for some $\lambda\in [0,\infty)$ and let $S = (V^*, E^*)$ be the subgraph of $G$ induced by $x^*$ and $k = e^T x^*$. Then,
\begin{enumerate}
\item[(1)]\label{one} $S$ is a sparsest $k$-subgraph of the graph $G$.
\item[(2)]\label{two} $\Delta(S)\leq\lambda$.
\item[(3)]\label{three} If $d_S(v) = \lambda$ for some $v\in V^*$, then $S-v$ is a sparsest $(k-1)$-subgraph of the graph $G$.
%\item[(4)]\label{four} 
%$S$ contains a maximum sparsest subgraph $S'$ of $G$ with $\Delta(S')<\lambda$ as an induced subgraph.\RK[This is incorrect as Janez pointed out. We should remove it. Although, if $\lambda$ is an irrational number this should hold. But this is rather a curiosity then something of a practical value.]
\end{enumerate}
\end{prop}

\begin{proof}

~
\begin{enumerate}
\item[(1)] This is essentially the claim of Lemma \ref{lem:optimality_test_LR}.

    \item[(2)] Suppose there exists a vertex $v \in V^*$ such that $d_S(v) > \lambda$. Let  $S - v$ denote the subgraph obtained by removing $v$ from $S$ and let $\tilde x$ be an incidence vector representing   $S-v$. We have $e^T\tilde x=k-1$ and the number of edges in the subgraphs defined by $\tilde x$ is the number of edges in $S$, reduced by  $d_S(v)$. Therefore, we have   
    \begin{align}\label{eq:Prop_8_2}
    \frac{1}{2}\tilde x^TA\tilde x-\lambda e^T\tilde x = 
    \frac{1}{2} x^{*T}A x^*-d_S(v) - \lambda e^T x^*+\lambda  < 
    \frac{1}{2} x^{*T}A x^* - \lambda e^T x^*,    
    \end{align}
      which  contradicts the assumption that $x^*$ minimizes~\eqref{SkS_QUBO_Lagrangian_1}. Hence, $\Delta(S) \leq \lambda$.
    
    \item[(3)] Let us introduce $\tilde x$ same was as in item (2).
    Now, the chain of inequalities in ~\eqref{eq:Prop_8_2} is actually the chain of equalities, since $d_S(v) = \lambda$. Therefore, the objective value of~\eqref{SkS_QUBO_Lagrangian_1} remains unchanged. This implies that $\tilde x$ is also a minimizer of ~\eqref{SkS_QUBO_Lagrangian_1}. By item (1), $S - v$ must be a sparsest $(k-1)$-subgraph of $G$.
\end{enumerate}
\end{proof}

\begin{cor}\label{lambda_equals_one}
Let $x^*$ be a minimizer of~\eqref{SkS_QUBO_Lagrangian_1} for $\lambda=1$ and let $S = (V^*, E^*)$ be the subgraph of $G$ induced by $x^*$ and $k = \vert V^* \vert$. Then, the subgraph $S'$ obtained from $S$ by iteratively removing all vertices with degree $1$ is a maximum stable set in $G$.   
\end{cor}

\begin{proof}
If $S$ is already a stable  set, then we are done. Otherwise, 
the second statement  of  Proposition \ref{roman:obs} implies that $\Delta(S) = 1$. Therefore, $S$ is a graph with  some isolated vertices (of degree 0) and with $|E^*|$ isolated edges. 
Therefore, if we remove from $S$ one end vertex for each edge from $E^*$, we get a graph $S'$ with $|V^*|-|E^*|$ vertices and no edge.
Therefore, the vertices in $S'$ are stable set in $G$. 
Let $\hat x$ be the incidence vector representing these vertices.
Then  $\hat x$ yields the same objective value of ~\eqref{SkS_QUBO_Lagrangian_1} as $x^*$, as removal of one end vertex from $S$  decreases both part of the objective function by 1.
The resulting set of vertices is a maximum stable set. Otherwise, any larger stable set would imply an incidence vector with lower objective value compared to $x^*$, which is not possible, since $x^*$ is optimum.
\end{proof}

\begin{lem}\label{complete_graph}
Let $G = (V, E)$ be a complete graph on $n$ vertices and let $k \leq n - 1$. Then setting $\lambda = k$ in the formulation~\eqref{SkS_QUBO_Lagrangian_1} yields an optimal solution $x^*$ that is the incidence vector of a sparsest $k$-subgraph or sparsest $(k+1)$-subgraph in $G$.
\end{lem}

\begin{proof}
Since we deal with a complete graph, we know that any subgraph on $k$ vertices is a sparsest $k$-subgraph in $G$. 
Thus, we have to show $e^Tx^*=k$. 
Note that, for any $k$-subgraph of $G$ is the  objective value of
\eqref{SkS_QUBO_Lagrangian_1} equal to 
$\frac{k(k-1)}{2} - k^2$, so we will prove that for any solution $\tilde x$ with $e^T\tilde x\neq k$, the objective value is larger than the objective value corresponding to any $k$-subset.

Let $\tilde{x} \in \{0, 1\}^n$ with $e^T\tilde{x} = k + i$ for some $i \in \{1, \ldots, n - k\}$. We have:
\begin{align}\label{eq:ineq1}
\frac{1}{2}\tilde x^{T} A \tilde x - k e^T\tilde x &=  \frac{(k+i)(k + i - 1)}{2} - k(k+i) \ge  
\frac{k(k-1)}{2} - k^2,
\end{align}
for all $i \in \{1, \ldots, n - k\}$. The last inequality is a simple calculus task. We can observe that this inequality becomes an equality only if $i=0,1$, hence we must have  $e^T x^*=k$ or $e^T x^*=k+1$, since otherwise any $k$ or $(k+1)$ subgraph would yield a better solution.
Similarly, if  $e^T\tilde x = k - j$ for some $j \in \{1, \ldots, k - 1\}$, we have 
\begin{align*}
\frac{1}{2}\tilde x^{T} A \tilde x - k e^T\tilde x =
 \frac{(k - j)(k - j - 1)}{2} - k(k - j) \ge \frac{k(k-1)}{2} - k^2
\end{align*}
for all $j \in \{1, \ldots, k - 1\}$.

Again, the last inequality is a straightforward calculation, and it is an equality only if $j=0$, hence we must have $e^T x^*=k$, since otherwise any $k$-subgraph would yield a better solution of ~\eqref{SkS_QUBO_Lagrangian_1}.

\end{proof}
Note that Lemma~\ref{complete_graph} trivially holds also for $k\in \{0,n\}$. 
%Additionally, for the complete graph $G$ with $n$ vertices, and for $\lambda = k < n-1$, any subset of $k$ or $k+1$ vertices gives the same value of the objective function in ~\eqref{SkS_QUBO_Lagrangian_1}, since for $i=1$ the inequality ~\eqref{eq:ineq1} becomes an equality.

\begin{thm}\label{roman:existence}
Let  $G$ be a graph with $n$ vertices and  $1\le k\le n-1$. Let us define 
$$
A_k := \max_{k^-<k}\left\{ \frac{m_{k}-m_{k^-}}{k-k^-}\right\}, B_k := \min_{k^+>k}\left\{\frac{m_{k^+}-m_k}{k^+-k}\right\}.
$$\\
\\
Then the following is true:
\begin{itemize}
    \item [(i)] If $x_k$ is an incidence vector of a sparsest $k$-subgraph in $G$, then  
    $x_k$ minimizes ~\eqref{SkS_QUBO_Lagrangian_1} with some $\lambda$ if and only if 
 $ A_k\leq \lambda \leq B_k$.
   \item[(ii)]
   Suppose $A_k<B_k$. 
   If $x$ is an optimum solution for ~\eqref{SkS_QUBO_Lagrangian_1} with some $\lambda \in (A_k,B_k)$, then $x$ is an incidence vector of a sparsest $k$-subgraph in $G$.
\end{itemize}

\end{thm}

\begin{proof}
We first prove (i).
Assume $x_k$  is an incidence vector of a sparsest $k$-subgraph in $G$ that minimizes ~\eqref{SkS_QUBO_Lagrangian_1} for some $\lambda$. Then we have:
\begin{align*}
m_k -\lambda k \le & ~ m_{k^-} - \lambda k^-,\quad \forall k^-<k,\\
m_k -\lambda k  \le & ~ m_{k^+} - \lambda k^+,\quad \forall k^+>k,     
\end{align*}
since for all $k^-$ and $k^+$ sparsest subgraphs 
containing $m_{k^-}$ and $m_{k^+}$ edges, respectively, give upper bounds for the optimum value of ~\eqref{SkS_QUBO_Lagrangian_1}.

We can rewrite the above equalities as:
$$
\frac{m_{k}-m_{k^-}}{k-k^-}\leq\lambda\leq \frac{m_{k^+}-m_k}{k^+-k}, \quad \forall k^-<k,\forall  k^+>k,
$$
hence $A_k\leq \lambda \leq B_k.$

This proves the ``only if'' part of equivalence in (i).
To prove the ``if part'' of this equivalence, let us assume that $ A_k\leq \lambda \leq B_k$ and $x_k$ is not optimum for ~\eqref{SkS_QUBO_Lagrangian_1}.
Therefore, there exists optimum solution  $\tilde x$ with $e^T\tilde x \neq k$. Without loss of generality, we can assume that    $e^T\tilde x=k_+>k$.
It follows that
$$m_k-\lambda k > m_{k_+}-\lambda k_+.$$
We can rewrite this as 
$ \lambda > \frac{m_{k_+}-m_k}{k_+-k}\ge B_k$, a contradiction.

To prove (ii), suppose  that $x$   is an optimum for ~\eqref{SkS_QUBO_Lagrangian_1} with 
$\lambda \in (A_k,B_k)$.
 If $e^Tx=k$, then the claim follows by Proposition \ref{roman:obs} (here we use a trivial observation that $A_k\ge 0$).
 If $e^Tx=\bar k\neq k$, we can assume without loss of generality that
  $\bar k > k$. 
Then  it follows:
    $\frac{1}{2}x^TAx-\lambda\bar k  \le  m_k-\lambda k$, therefore 
  $$\lambda \ge  \frac{\frac{1}{2}x_k^TAx_k-m_k}{\bar k-k}\ge \frac{m_{\bar k}-m_k}{\bar k - k} \ge B_k,$$
  a contradiction.
\end{proof}

Now we consider what happens if the conditions from Theorem \ref{roman:existence} are not satisfied.

\begin{lem} \label{lem: A_k_>B_k}
    Let $G$ be a graph on $n$ vertices, let $k \leq n - 1$, and suppose that $A_k > B_k$. Then, for any value of $\lambda$, the optimum solution of ~\eqref{SkS_QUBO_Lagrangian_1} will not have cardinality $k$.
\end{lem}

\begin{proof}
    Assume that there exists a $\lambda$, for which the optimum solution of ~\eqref{SkS_QUBO_Lagrangian_1}, which we denote by $x^*$, has cardinality $k$. Let $m_k$ be the number of edges in the sparsest $k$-subgraph corresponding to $x^*$.
    
    This implies that $\forall k^- < k$ we have the following equivalent inequalities:
    \begin{align*}
        m_k - \lambda k &\le m_{k^-} - \lambda k^-  \\
        m_k - m_{k^-} &\le \lambda (k - k^-) \\
        \frac{m_k - m_{k^-}}{k - k^-} & \le \lambda .
    \end{align*}
    Therefore, it follows $ \max_{k^- < k} \left( \frac{m_k - m_{k^-}}{k - k^-} \right) = A_k \le \lambda$. 
    
    Similarly,  $\forall k^+ > k$ we have
    \begin{align*}
        m_k - \lambda k &\le m_{k^+} - \lambda k^+  \\
        \lambda (k^+ - k) &\le m_{k^+} - m_k \\
        \lambda &\le \frac{m_{k^+} - m_k}{k^+ - k} .
    \end{align*}

    This implies that $ \lambda \le \min_{k^+ > k} \left( \frac{m_{k^+} - m_k}{k^+ - k} \right) = B_k$. If we combine both results, we get  $A_k \le \lambda \le B_k$, hence $ A_k \le B_k$, a contradiction. Hence, there does not exist any $\lambda$ for which an optimal solution of ~\eqref{SkS_QUBO_Lagrangian_1} has  cardinality $k$.
\end{proof}

We now show that, for certain instances, it is possible to select tighter bounds on $\lambda$. 
To this end, we revisit 
$\{\diff_\ell\}$, introduced in Section~\ref{section_SkS_QUBO_exact_penalty}, and demonstrate that if $\{\diff_\ell\}$ is monotonically increasing in a given graph, then the formula for  $A_k$ and $B_k$ from Theorem \ref{roman:existence} can be simplified.
  Additionally, we can choose $\lambda$ to be an integer, further emphasizing the practical effectiveness of Lagrangian relaxation for this problem.\\

\begin{cor} \label{cor: diff_k>diff_k_plus_1}
    Let $G$ be a graph on $n$ vertices, let $k \leq n - 1$, and suppose that $\diff_k > \diff_{k+1}$ for some $k$. Then for any value of $\lambda$, the optimum solution of ~\eqref{SkS_QUBO_Lagrangian_1} will not have cardinality $k$.
\end{cor}

\begin{proof}
    By assumption we  have $\diff_{k} > \diff_{k+1}$. Since $\diff_{k} \le A_{k}$ and $B_{k} \le \diff_{k+1}$, this implies $A_{k} > B_{k}$. Hence from Lemma \ref{lem: A_k_>B_k}, there does not exist any $\lambda$ for which the optimum solution of ~\eqref{SkS_QUBO_Lagrangian_1} will have cardinality $k$.
\end{proof}

\begin{lem} \label{lem: MIS_for_lam}
    Let $G$ be a graph on $n$ vertices, and let $\alpha$ be the size of the maximum stable set of $G$. Let $m_{\alpha+1}$ be the number of edges in the sparsest $k$-subgraph for $k = \alpha+1$. Then, for $\lambda < m_{\alpha + 1}$, the optimum solution of ~\eqref{SkS_QUBO_Lagrangian_1} will correspond to a maximum stable set.
\end{lem}

\begin{proof}
We have $A_{\alpha} = 0$ and $B_{\alpha} \le \diff_{\alpha+1} = m_{\alpha+1}$. Also, $A_{\alpha} < B_{\alpha}$. Hence, according to Theorem \ref{roman:existence}, for $\lambda \in (A_{\alpha}, B_{\alpha}) = (0,m_{\alpha+1})$, the optimum solution of ~\eqref{SkS_QUBO_Lagrangian_1} will correspond to a maximum stable set.
\end{proof}

Note that Lemma \ref{lem: MIS_for_lam} can be considered as a generalization of  Lemma \ref{stable_set_lagrangian}.

\begin{lem}\label{lambda_tighter_bounds}
Let $G$ be a graph on $n$ vertices, let $k \leq n - 1$, and suppose that $\{\diff_\ell\}$ is monotonically increasing in $G$. Then, 
$$
A_k = \diff_k \le \diff_{k+1}  = B_k.
$$
\end{lem}

\begin{proof}
First, we prove $
A_k = \diff_k$.

Since $\{\diff_\ell\}$ is monotonically increasing, we have for all $k^-<k$ that  
$$m_k-m_k^- \le m_k-m_{k-1}+m_{k-1}-m_{k-2}+\cdots+m_{k^-+1}-m_{k^-} \le \diff_k(k-k^-), $$
hence 
$    A_k \le  \diff_k.$
On the other hand, we know that $\diff_k = \frac{m_k-m_{k-1}}{k-(k-1)}$ and 
$A_k= \max_{k^{-}<k}\frac{m_k-m_{k^-}}{k-k^-} \ge \frac{m_k-m_{k-1}}{k-(k-1)} = \diff_k $. hence $A_k\ge \diff_k$, hence we have the equality.

Similarly, we show that $B_k = \diff_{k+1}$.
\end{proof}

The following Corollary trivially follows from Theorem \ref{roman:obs} and Lemma \ref{lambda_tighter_bounds}.
\begin{cor}\label{corollary_LR_diff_l}
Let $G$ be a graph on $n$ vertices and suppose that $\{\diff_\ell\}$ is monotonically increasing in $G$. If for some  $k \in \{1, \ldots, n-1\}$ we have $\diff_k < \diff_{k+1}$, then  any optimum solution of~\eqref{SkS_QUBO_Lagrangian_1} with $\lambda \in (\diff_k,\diff_{k+1})$ is a sparsest $k$-subgraph in $G$.
\end{cor}

We continue with a theorem that describes how sensitive the number of optimal solutions is to the value of $\lambda$.

\begin{thm} \label{thm: boundary_ks}
Let $G = (V, E)$ be a graph on $n$ vertices. 
Let $k_1$ and $k_2$ be  the smallest and largest integer values, respectively, such that $\alpha(G) \leq k_1< k_2 \leq n$, and the corresponding incidence vectors $x_{k_1}$ and $x_{k_2}$ with $e^Tx_{k_i}=k_i$ both minimize  ~\eqref{SkS_QUBO_Lagrangian_1} for some $\lambda \in \mathbb{R}^+$.

Then, there exist $d^- > 0$ and $d^+ > 0$ such that:  
\begin{itemize}
    \item For all $\epsilon_1 \in (0, d^-)$, every minimizer of~\eqref{SkS_QUBO_Lagrangian_1} for $\lambda - \epsilon_1$ has cardinality exactly $k_1$.
    \item For all $\epsilon_2 \in (0, d^+)$, every minimizer of~\eqref{SkS_QUBO_Lagrangian_1} for $\lambda + \epsilon_2$ has cardinality exactly $k_2$.
\end{itemize} 
\end{thm}

\begin{proof}
Since $e^Tx_{k_i}=k_i$, we have that $x_{k_1}$ and $x_{k_2}$  are incidence vectors of  sparsest  $k_1$ and $k_2$-subgraphs, respectively.
 Let $ m_{k_1}$ and  $ m_{k_2}$ be the number of edges corresponding to $x_{k_1}$ and $x_{k_2}$, respectively.
We have 
\begin{align*}
    m_{k_1}-\lambda k_1 =  m_{k_2} -\lambda k_2. 
\end{align*}

We have to show that for $\varepsilon$ small enough, the value $(\lambda-\epsilon)$ implies
\begin{align*}
    m_{k_1} - (\lambda-\epsilon)k_1 < m_{k} - (\lambda-\epsilon)k,\;\; \forall k  \neq k_1.
\end{align*} 
For $k > k_1$ and $\varepsilon >0$, suppose  
\begin{align*}
m_{k_1} - (\lambda-\epsilon)k_1 &\ge  m_{k} - (\lambda-\epsilon)k.
\end{align*}
Then we get a chain  of  inequalities: 
\begin{align*}
0\le 
m_{k_1} - \lambda k_1 + \epsilon k_1 -(  m_{k} - \lambda k + \epsilon k) =
     (m_{k_1} - \lambda k_1) - (m_{k} - \lambda k)+ \epsilon k_1 - \epsilon k \le 
    \epsilon (k_1 - k). 
\end{align*}
therefore $k_1 - k\ge 0$, which is a  
contradiction. Note that the last inequality in the above chain follows from the fact that $x_{k_1}$ is optimal for ~\eqref{SkS_QUBO_Lagrangian_1}.

Suppose  now $k < k_1$ and $\epsilon>0$.
We have $m_{k_1} - \lambda k_1 < m_{k} - \lambda k$. Let $d' = \lambda(k_1-k) - (m_{k_1} - m_k) > 0$. The following  inequalities are equivalent:
\begin{align*}
    m_{k_1} - (\lambda - \epsilon)k_1 \ge & ~m_k - (\lambda - \epsilon)k \\
    0 \ge & ~ (\lambda - \epsilon)(k_1-k) - (m_{k_1} - m_k)\\
    0 \ge & ~ d' - \epsilon(k_1-k)\\
    \epsilon \ge  & ~ \frac{d'}{k_1 - k}.
\end{align*}
Therefore, if we set 
$  d^-:=
 \min_{k<k_1} \frac{d'}{k_1 - k} = \min_{k<k_1} \left( \lambda - \frac{m_{k_1} - m_{k}}{k_1 - k} \right) > 0$, then for $0<\epsilon<d^-$, we have $m_{k_1} - (\lambda-\epsilon)k_1 < m_{k} - (\lambda-\epsilon)k,\; \forall k < k_1$.

Similarly, we can show for $0<\epsilon<d^+=\min_{k>k_2} \left( \frac{m_{k_2} - m_{k}}{k_2 - k} - \lambda \right)$, every minimizer of ~\eqref{SkS_QUBO_Lagrangian_1} will have cardinality $k_2$.
\end{proof}

\subsection{Augmented Lagrangian relaxations for \ref{SkS_problem_standard_formulation}}
\label{augmented_lagrangian_SkS}

As a third approach, we use the augmented Lagrangian method, which combines the quadratic penalty approach with Lagrangian relaxation. This leads to the following QUBO relaxation for the ~\eqref{SkS_problem_standard_formulation} problem:

\begin{align}\label{SkS_QUBO_formulation_two}
\tag{\mbox{SkS-AL}}
     \min \left\{\frac{1}{2}x^TAx + \lambda(k - e^Tx) + \frac{\mu}{2}(e^Tx - k)^2 
     \mid x \in \{0, 1\}^n \right\},
\end{align}
where $\mu$ is referred to as the penalty parameter, and $\lambda$ as the Lagrangian multiplier; see, for instance~\cite{Bertsekas}.

As with the Lagrangian relaxation, the values of $\mu$ and $\lambda$ in the augmented Lagrangian method are generally not known in advance and are therefore typically determined through an iterative procedure, where  solve ~\eqref{SkS_QUBO_formulation_two} with
selected  $\mu,\lambda$ and then update these parameters, until stopping criteria are reached. 

The central stopping criteria is feasibility of the computed solution for the primal constraint. Indeed,  if we have values of $\lambda$ and $\mu$ that yield a solution of desired cardinality, then we have  an optimal solution to the original problem, as the following result shows.\\
\begin{lem}\label{lem:optimality_test_ALM}
Let $G$ be a graph on $n$ vertices and $k \leq n$. Furthermore, let $x^*$ be an optimal solution of~\eqref{SkS_QUBO_formulation_two} for some $\lambda \in \mathbb{R}$ and $\mu > 0$. If $e^T x^* = k$, then $x^*$ is also an optimal solution of the original  problem~\eqref{SkS_problem_standard_formulation}.
\end{lem}

\begin{proof}
We proceed in the same manner as in the proofs of Lemmas~\ref{lem:optimality_test_QUBO} and~\ref{lem:optimality_test_LR}. First, given that $x^*$ is optimal for~\eqref{SkS_QUBO_formulation_two}, we note that
\begin{align*}
\frac{1}{2}x^{*T} A x^* + \lambda (k - e^T x^*) + \frac{\mu}{2}(e^Tx^* - k)^2 \leq \frac{1}{2} x^T A x + \lambda (k - e^T x) + \frac{\mu}{2}(e^Tx - k)^2
\end{align*}
for all $x \in \{0,1\}^n$. Now let $e^T x^* = k$. Then, 
\begin{align*}
\frac{1}{2} x^{*T} A x^* \leq \frac{1}{2} x^T A x + \lambda (k - e^T x) + \frac{\mu}{2}(e^Tx - k)^2
\end{align*}
for all $x \in \{0,1\}^n$. Moreover,
\begin{align*}
\frac{1}{2}x^{*T} A x^* \leq \frac{1}{2} x^T A x 
\end{align*}
for all $x \in \{0,1\}^n$ with $e^T x = k$. Therefore, $x^*$ is also optimal  for~\eqref{SkS_problem_standard_formulation}.
\end{proof}

As discussed in Section~\ref{QUBO_theory}, there is no theoretical recipe for selecting the penalty parameter and the Lagrangian multiplier in general, but for the case of the \ref{SkS_problem_standard_formulation} problem, we can always choose in advance parameters that 
yield an optimal solution.\\

\begin{lem}\label{lemma_ALM_SkS}
Let $G$ be a graph on $n$ vertices and $k \leq n$. Let  $\lambda = \frac{1}{2}(k-1)$ and  $\mu = k$.
Then,
\begin{itemize}
    \item [(i)]
    If  $x^* \in \{0, 1\}^n$ is  an optimal solution of ~\eqref{SkS_problem_standard_formulation}, then
\begin{align}\label{ALM_SkS}
\frac{1}{2} x^{*T} A x^* < \frac{1}{2}x^TAx + \lambda(k - e^T x) + \frac{\mu}{2}(e^Tx - k)^2
\end{align}
for all $x \in \{0, 1\}^n$ with $e^Tx \neq k$. 
\item [(ii)] Any optimal solution of~\eqref{SkS_QUBO_formulation_two}  is an optimal solution of ~\eqref{SkS_problem_standard_formulation}.
\end{itemize}
\end{lem} 

\begin{proof}
We first prove {\it (i)}. Let $x \in \{0, 1\}^n$ such that $e^T x = k + i$ for some $i \in \{1, \ldots, n-k\}$. Then, 
\begin{align*}
\frac{1}{2}x^TAx + \lambda(k - e^Tx) + \frac{\mu}{2}(e^Tx - k)^2 &= \frac{1}{2}x^TAx - \frac{1}{2}(k-1) i + \frac{k}{2} i^2 \\
&\geq m_{k+i} - \frac{1}{2}(k-1) i + \frac{k}{2} i^2 \\
&\geq m_{k} - \frac{1}{2}(k-1) i + \frac{k}{2} i^2 \\
&> m_k = \frac{1}{2} x^{*T} A x^*,
\end{align*}
because $m_{k+i}\ge m_k$ and $- \frac{1}{2}(k-1) i + \frac{k}{2} i^2 > 0$ for any $i \in \{1, \ldots, n-k\}$. Hence,~\eqref{ALM_SkS} is satisfied for any $x \in \{0, 1\}^n$ with $e^T x > k$.

Now let $x \in \{0, 1\}^n$ such that $e^T x = k - j$ for some $j \in \{1, \ldots, k - 1\}$. Then, following the argumentation from Section~\ref{section_SkS_QUBO_exact_penalty}, we obtain
\begin{align*}
\frac{1}{2}x^TAx + \lambda(k - e^Tx) + \frac{\mu}{2}(e^Tx - k)^2 &= \frac{1}{2}x^TAx + \frac{1}{2}(k-1) j + \frac{k}{2} j^2 \\
&\geq m_{k-j} + \frac{1}{2}(k-1) j + \frac{k}{2} j^2 \\
&> m_{k-j} + kj - \frac{j(j+1)}{2} \\
&\geq m_k = \frac{1}{2} x^{*T} A x^*.
\end{align*}
The strict inequality above follows from the fact that
since $\frac{1}{2}(k-1) j + \frac{k}{2} j^2 > kj - \frac{j(j+1)}{2}$ for any $j \in \{1, \ldots, k - 1\}$. The last inequality is essentially ~\eqref{eq:m_k_2}.
Thus,~\eqref{ALM_SkS} also holds for any $x \in \{0, 1\}^n$ with $e^T x < k$ and {\it (i)} is proven.

To prove {\it (ii)}, let us assume that $\hat x$ is optimal for ~\eqref{SkS_QUBO_formulation_two} and $x^*$ is optimal for ~\eqref{SkS_problem_standard_formulation}. If $e^T\hat x\neq k$, then {\it (i)} implies
\begin{align*}
\frac{1}{2} \hat{x}^T A \hat{x} + \lambda(k - e^T\hat{x}) + \frac{\mu}{2}(e^T\hat{x} - k)^2 > \frac{1}{2} x^{*T} A x^*,
\end{align*}
contradicting the assumption of optimality of $\hat x$ for ~\eqref{SkS_QUBO_formulation_two}.
Thus, $e^T\hat{x} =  k$ and $\hat x$ is feasible for ~\eqref{SkS_problem_standard_formulation}.
If it is not an optimal solution for~\eqref{SkS_problem_standard_formulation}, then 
we have
\begin{align*}
\frac{1}{2} \hat x^T A \hat x > \frac{1}{2} x^{*T} A x^*
\end{align*}
which again contradicts the optimality  of $\hat x$ for ~\eqref{SkS_QUBO_formulation_two}. Hence, any optimal solution of~\eqref{SkS_QUBO_formulation_two} is an optimal solution of ~\eqref{SkS_problem_standard_formulation}, which finishes the proof.
\end{proof}

\section{Numerical Results}

In the numerical section, we first introduce the graph instances, for which we solve SkS exactly or approximately by the methods implied by the results from  Section \ref{definitions}. 

Next, we present results obtained by applying the exact solver BiqBin \cite{gusmeroli2022biqbin}, available from 
\url{https://github.com/Rudolfovoorg/parallel_biqbin_maxcut}, on quadratic, Lagrangian, and Augmented Lagrangian relaxation of ~\eqref{SkS_problem_standard_formulation}. For the penalty parameters, we used the values for which we have proven in Section \ref{definitions} that   an optimum solution of the relaxation is an optimum solution of ~\eqref{SkS_problem_standard_formulation}.
Due to the underlying NP-hardness of SkS problem, we could use this approach only for the instances of small or medium size.

Finally, we introduce three approximate algorithms based on quadratic penalty, Lagrangian relaxation, and augmented Lagrangian relaxation methods, which in their core use a simulated annealing  and 
 a quantum processing unit solver to get good feasible solutions of the relaxations of ~\eqref{SkS_problem_standard_formulation}.
 This approach also works for larger instances.

\subsection{Data}\label{sec:Data}
For the computations, we use three datasets, which are all published as open data in \cite{Github_Data} 

\begin{enumerate}
\item Random graphs from the Erdős–Rényi (ER) model~$G(n, p)$. In this model, the number of vertices~$n$ is fixed, and each edge is included with probability~$p$ independently of all other edges. We consider graphs with $n \in \{40,80,100,120,140, 160\}$ and $p \in \{0.25, 0.50, 0.75 \}$ that were previously investigated in~\cite{Billionnet2005, Billionnet2009, Sotirov2020}, and that are available from the following webpage: \url{https://cedric.cnam.fr/~lamberta/Library/k-cluster.html}.
    For each $n$ and $p$, there are 3 subgroups of 5 instances: one subgroup with the optimum $k=n/4$, one with optimum $k=n/2$, and one with optimum $k=3n/4$,   resulting in overall 
270 instances.
    
    For these graphs, the webpage \url{https://cedric.cnam.fr/~lamberta/Library/solutions_k-cluster.html} contains the optimum values for the densest $k$-subgraph, which is equivalent to SkS problem on the graph complement, as explained in Section \ref{definitions}.

    For easier comparison of our results with the results from the source, we decided to report results for the densest $k$-subgraph, which means that we first created a complement of each graph, solved SkS problem on this complement to obtain a sparsest $k$-subgraph, and finally calculated its complement to obtain a densest $k$-subgraph.

    \item Bipartite graphs $BG_{m,n,p}$ with partition
    sizes $m$ and $n$ and with edge probability $p$. These graphs trivially satisfy assumptions of Theorem \ref{roman:existence}, as explained below, so we computed optimum solutions of ~\eqref{SkS_problem_standard_formulation}
    % by BiqBin solver, 
    using Lagrangian relaxation.

    \item We created a D-Wave topology graphs, derived from the actual topology of the D-Wave Advantage2 system (Zephyr architecture) hosted at Forschungszentrum Jülich. For sizes $n=50, 75, 100, 200, 300, \ldots, 1000, 1500$ we extracted the dense subgraphs on $n$ vertices  using a greedy algorithm, which iteratively removed vertices with the highest degree, until only $n$ vertices remained. These subgraph are denoted  
    by $DW_{50},\ldots,DW_{1500}$.
\end{enumerate}

\subsection{Results with Exact Solver}\label{sec:exact_solvers}
As a first step, we report results obtained by solving 
\eqref{SkS_QUBO_formulation_one}, ~\eqref{SkS_QUBO_Lagrangian_1}, and ~\eqref{SkS_QUBO_formulation_two} exactly, by using the values of penalty parameters, as suggested by lemmas and theorems from Section \ref{definitions}.
We use the ER dataset for the quadratic penalty and the augmented Lagrangian approach, while the Lagrangian relaxation approach is tested on bipartite graphs, since they are designed to satisfy assumptions of Theorem \ref{roman:existence}, which are needed for this 
approach, while for the other two datasets we do not know if they satisfy these assumptions. 
For the quadratic penalty and augmented Lagrangian approaches, we employ the BiqBin solver \cite{gusmeroli2022biqbin} on the Slovenian national supercomputer Vega\footnote{https://izum.si/en/vega-en/}, using up to 1000 CPU cores,  with a time limit of 30 minutes per instance. For the Lagrangian relaxation approach, solutions are obtained with the SCIP solver \cite{SCIP} and laptop.

\subsubsection{Exact solutions with Quadratic Penalty}
Here, we report numerical results obtained by solving ~\eqref{SkS_QUBO_formulation_one} for the ER graphs. More precisely, we compute results for the densest $k$-subgraph problem, which is equivalent to the sparsest $k$-subgraph problem on the graph complement, as explained in Section \ref{sec:Data}. As these calculations are computationally very extensive, we have not carried them out for all 270 instances, but only for 9 of them: for each $n\in\{40,80,100\}$ and each $p\in\{0.25,0.5,0.75\}$ we have taken only the subgroup with $k=n/4$ and only the first instances of each subgroup. 

For the penalty parameter, we initially take the value
$\mu_{LB} = 2\min\{\tilde{m}_k,k-1\}+1$, as suggested in ~\eqref{choose_quadratic_penalty}.
The values of $\tilde{m}_k$ are computed by a greedy algorithm, which iteratively removes vertices with the highest degree, until only $k$ vertices remain.

We hypothesize that values of $\mu$ smaller than $\mu_{LB}$ may still yield optimum solutions; therefore we also solve relaxation  ~\eqref{SkS_QUBO_formulation_one} for $\mu = \lfloor \mu_{LB}/3 \rfloor$ and $\mu = \lfloor 2\mu_{LB}/3 \rfloor$.
For the sake of curiosity, we additionally consider a larger value, $\mu = \lfloor 4\mu_{LB}/3 \rfloor$. The corresponding results are presented in Table~\ref{tab: QP_biqbin}. 

\begin{table}[]
    \centering
    \resizebox{\linewidth}{!}{%
\begin{tabular}{lccclcccc}
\hline
\textbf{Instance} & $\boldsymbol{k}$ & \textbf{Optimum} & $\boldsymbol{\mu_{LB}}$ & \textbf{Metric} & $\boldsymbol{\lfloor \frac{\mu_{LB}}{3} \rfloor} $ 
& $\boldsymbol{\lfloor \frac{2\mu_{LB}}{3} \rfloor}$ 
& $\boldsymbol{\mu_{LB}}$ 
& $\boldsymbol{\lfloor \frac{4\mu_{LB}}{3} \rfloor}$ \\
\hline
\multirow[t]{6}{*}{kcluster40\_025\_10\_1} & \multirow[t]{6}{*}{10} & \multirow[t]{6}{*}{29} & \multirow[t]{6}{*}{19} & $\mu$ & 6 & 12 & 19 & 25 \\
 &  &  &  & time & 1.281 & 5.224 & 1.711 & 0.422 \\
 &  &  &  & B\&B nodes & 1 & 11 & 1 & 1 \\
 % &  &  &  & B\&B depth & 0 & 4 & 0 & 0 \\
 &  &  &  & computed $k$ & 9 & 10 & 10 & 10 \\
 &  &  &  & edges & 24 & \textbf{29} & \textbf{29} & \textbf{29} \\
\cline{1-9} \cline{2-9} \cline{3-9} \cline{4-9}
\multirow[t]{6}{*}{kcluster40\_050\_10\_1} & \multirow[t]{6}{*}{10} & \multirow[t]{6}{*}{40} & \multirow[t]{6}{*}{15} & $\mu$ & 5 & 10 & 15 & 20 \\
 &  &  &  & time & 0.746 & 0.576 & 0.533 & 0.491 \\
 &  &  &  & B\&B nodes & 1 & 1 & 1 & 1 \\
 % &  &  &  & B\&B depth & 0 & 0 & 0 & 0 \\
 &  &  &  & computed $k$ & 10 & 10 & 10 & 10 \\
 &  &  &  & edges & \textbf{40} & \textbf{40} & \textbf{40} & \textbf{40} \\
\cline{1-9} \cline{2-9} \cline{3-9} \cline{4-9}
% \multirow[t]{6}{*}{kcluster40\_075\_10\_1} & \multirow[t]{6}{*}{10} & \multirow[t]{6}{*}{45} & \multirow[t]{6}{*}{1} & $\mu$ & 0 & 0 & 1 & 1 \\
%  &  &  &  & time & 0.555 & 0.604 & 0.794 & 0.781 \\
%  &  &  &  & B\&B nodes & 1 & 1 & 1 & 1 \\
%  % &  &  &  & B\&B depth & 0 & 0 & 0 & 0 \\
%  &  &  &  & computed $k$ & 0 & 0 & 10 & 10 \\
%  &  &  &  & edges & 0 & 0 & \textbf{45} & \textbf{45} \\
\multirow[t]{6}{*}{kcluster40\_075\_10\_1} & \multirow[t]{6}{*}{10} & \multirow[t]{6}{*}{45} & \multirow[t]{6}{*}{1} & $\mu$ & 0 & 0 & 1 & 1 \\
 &  &  &  & time & - & - & 0.794 & 0.781 \\
 &  &  &  & B\&B nodes & - & - & 1 & 1 \\
 % &  &  &  & B\&B depth & 0 & 0 & 0 & 0 \\
 &  &  &  & computed $k$ & - & - & 10 & 10 \\
 &  &  &  & edges & - & - & \textbf{45} & \textbf{45} \\
\cline{1-9} \cline{2-9} \cline{3-9} \cline{4-9}
\multirow[t]{6}{*}{kcluster80\_025\_20\_1} & \multirow[t]{6}{*}{20} & \multirow[t]{6}{*}{94} & \multirow[t]{6}{*}{39} & $\mu$ & 13 & 26 & 39 & 52 \\
 &  &  &  & time & 3.341 & 420.747 & 82.4 & 50.131 \\
 &  &  &  & B\&B nodes & 5 & 10169 & 1569 & 435 \\
 % &  &  &  & B\&B depth & 1 & 53 & 25 & 22 \\
 &  &  &  & computed $k$ & 19 & 20 & 20 & 20 \\
 &  &  &  & edges & 87 & \textbf{94} & \textbf{94} & \textbf{94} \\
\cline{1-9} \cline{2-9} \cline{3-9} \cline{4-9}
\multirow[t]{6}{*}{kcluster80\_050\_20\_1} & \multirow[t]{6}{*}{20} & \multirow[t]{6}{*}{149} & \multirow[t]{6}{*}{39} & $\mu$ & 13 & 26 & 39 & 52 \\
 &  &  &  & time & 29.561 & 16.272 & 5.67 & 3.025 \\
 &  &  &  & B\&B nodes & 85 & 17 & 3 & 1 \\
 % &  &  &  & B\&B depth & 14 & 6 & 0 & 0 \\
 &  &  &  & computed $k$ & 20 & 20 & 20 & 20 \\
 &  &  &  & edges & \textbf{149} & \textbf{149} & \textbf{149} & \textbf{149} \\
\cline{1-9} \cline{2-9} \cline{3-9} \cline{4-9}
\multirow[t]{6}{*}{kcluster80\_075\_20\_1} & \multirow[t]{6}{*}{20} & \multirow[t]{6}{*}{182} & \multirow[t]{6}{*}{27} & $\mu$ & 9 & 18 & 27 & 36 \\
 &  &  &  & time & 22.238 & 7.333 & 6.564 & 5.584 \\
 &  &  &  & B\&B nodes & 47 & 29 & 27 & 21 \\
 % &  &  &  & B\&B depth & 12 & 4 & 5 & 3 \\
 &  &  &  & computed $k$ & 20 & 20 & 20 & 20 \\
 &  &  &  & edges & \textbf{182} & \textbf{182} & \textbf{182} & \textbf{182} \\
\cline{1-9} \cline{2-9} \cline{3-9} \cline{4-9}
\multirow[t]{6}{*}{kcluster100\_025\_25\_1} & \multirow[t]{6}{*}{25} & \multirow[t]{6}{*}{139} & \multirow[t]{6}{*}{49} & $\mu$ & 16 & 32 & 49 & 65 \\
 &  &  &  & time & 9.708 & 6554.86 & 404.633 & 158.865 \\
 &  &  &  & B\&B nodes & 19 & 138923 & 6765 & 2057 \\
 % &  &  &  & B\&B depth & 3 & 79 & 35 & 27 \\
 &  &  &  & computed $k$ & 24 & 25 & 25 & 25 \\
 &  &  &  & edges & 130 & \textbf{139} & \textbf{139} & \textbf{139} \\
\cline{1-9} \cline{2-9} \cline{3-9} \cline{4-9}
\multirow[t]{6}{*}{kcluster100\_050\_25\_1} & \multirow[t]{6}{*}{25} & \multirow[t]{6}{*}{218} & \multirow[t]{6}{*}{49} & $\mu$ & 16 & 32 & 49 & 65 \\
 &  &  &  & time & 132.443 & 64.37 & 48.963 & 44.378 \\
 &  &  &  & B\&B nodes & 1287 & 575 & 307 & 193 \\
 % &  &  &  & B\&B depth & 39 & 24 & 19 & 18 \\
 &  &  &  & computed $k$ & 24 & 25 & 25 & 25 \\
 &  &  &  & edges & 203 & \textbf{218} & \textbf{218} & \textbf{218} \\
\cline{1-9} \cline{2-9} \cline{3-9} \cline{4-9}
\multirow[t]{6}{*}{kcluster100\_075\_25\_1} & \multirow[t]{6}{*}{25} & \multirow[t]{6}{*}{282} & \multirow[t]{6}{*}{43} & $\mu$ & 14 & 28 & 43 & 57 \\
 &  &  &  & time & 11.886 & 7.410 & 9.286 & 8.041 \\
 &  &  &  & B\&B nodes & 29 & 25 & 31 & 19 \\
 % &  &  &  & B\&B depth & 5 & 4 & 4 & 3 \\
 &  &  &  & computed $k$ & 25 & 25 & 25 & 25 \\
 &  &  &  & edges & \textbf{282} & \textbf{282} & \textbf{282} & \textbf{282} \\
\cline{1-9} \cline{2-9} \cline{3-9} \cline{4-9}
\hline
\end{tabular}
    }
    \caption{Performance of BiqBin solver on \eqref{SkS_QUBO_formulation_one} using different values of $\mu$.}
    \label{tab: QP_biqbin}
\end{table}

Table \ref{tab: QP_biqbin} is divided into sections by rows, one for each ER instance that we used. Each row section contains in the first column the name of the corresponding ER instance, in the second column the value of $k$ for which we know the optimum value of the densest $k$-subgraph problem from the literature,  in the third column the number of edges in the densest $ k$-subgraph, and in the fourth column the value of $\mu_{LB}$.

The rest of each row section contains  four numerical columns, one column for each value of $\mu$ that we used, and  five rows:
the first row contains explicit values of $\mu$ that were used, the second row contains timing that BiqBin solver needed to solve ~\eqref{SkS_QUBO_formulation_one} with corresponding $\mu$, the third row contains the number of Branch and Bound nodes that the BiqBin solver created before it has converged, the fourth row reports the cardinality of the optimum solution, and the last row of each section contains the number of the edges in the computed subgraph.

We can observe from Table \ref{tab: QP_biqbin} that on the test instances, for $\mu=\lfloor 2\mu_{LB}/3 \rfloor$, $\mu=\mu_{LB}$, and $\mu = \lfloor 4\mu_{LB}/3 \rfloor$, the computed optimum solution of ~\eqref{SkS_QUBO_formulation_one}
is always an optimum solution for the original problem ~\eqref{SkS_problem_standard_formulation}: its cardinality (computed $k$) is equal to $k$, and   the number of edges in the subgraph induced by the computed solution is equal to Optimum.
Note that for the third instance we have   $ \lfloor 2\mu_{LB}/3 \rfloor=
\lfloor \mu_{LB}/3 \rfloor=0$, hence in these two cases we lose the quadratic penalty term from ~\eqref{SkS_QUBO_formulation_one}, so we get a different optimization problem, hence we ignore these cases.  

\subsubsection{Exact solutions with Lagrangian Relaxation}
In this subsection,  we report results obtained by solving the  Lagrangian relaxation ~\eqref{SkS_QUBO_Lagrangian_1} with the exact solver BiqBin. 
While Lemma \ref{lem:optimality_test_LR} guarantees that if an optimal solution $x^*$ of ~\eqref{SkS_QUBO_Lagrangian_1} satisfies the cardinality constraint $e^T x^* = k$, then it is also optimal for the original problem ~\eqref{SkS_problem_standard_formulation}, the challenge lies in determining a suitable value for the Lagrangian multiplier $\lambda$ a priori that ensures this condition. Unlike the quadratic penalty and augmented Lagrangian approach, where specific bounds for the penalty parameters ($\lambda$, $\mu$) can be derived to guarantee exactness for general cases, a universally determined value of $\lambda$ for Lagrangian relaxation approach is not readily available for general instances of  SkS problem.

%We report numerical results which demonstrate what happens if $\lambda$ satisfies these assumptions and if it does not.

% As noted in Section \ref{section_LR_SkS}, for $\lambda \in (0, 1)$, any optimal solution of ~\eqref{SkS_QUBO_Lagrangian_1} is the incidence vector of a maximum stable set. This implies that if the target cardinality $k$ is less than or equal to the stability number $\alpha(G)$, we can obtain a sparsest $k$-subgraph by simply taking any $k$ vertices from the maximum stable set. However, this specific range of $\lambda$ does not guarantee a solution with an arbitrary target cardinality $k$.

Recall that  if $\lambda$  satisfies the assumptions of Theorem \ref{roman:existence}, then any optimal solution $x^*$ of ~\eqref{SkS_QUBO_Lagrangian_1} is also optimal for the original problem ~\eqref{SkS_problem_standard_formulation}.
%The behavior of the sequence $\{\diff_\ell\}$, defined as $\diff_\ell := m_\ell - m_{\ell-1}$, plays a significant role in determining the effectiveness of Lagrangian Relaxation with exact solvers. Specifically, if the sequence $\{\diff_\ell\}$ is monotonically increasing in $G$, then, as shown in 
Lemma \ref{lambda_tighter_bounds} implies that if the sequence $\{\diff_\ell\} $ is monotonic and  $\diff_k < \diff_{k+1}$ for some $k$, then any $\lambda \in (\diff_k, \diff_{k+1})$ will yield a sparsest $k$-subgraph (Corollary \ref{corollary_LR_diff_l}).

Unfortunately, without solving  SkS
explicitly for each $k$, we can not determine $\{\diff_\ell\mid 1\le \ell\le n \}$ and can not confirm that this sequence is monotonically increasing, so this theory can  not be applied   for solving 
\eqref{SkS_problem_standard_formulation} via ~\eqref{SkS_QUBO_Lagrangian_1}.

If for some graph $G$ and some $1\le k\le n-1$ we have
$\diff_k<\diff_{k+1}$, this is not sufficient condition in terms of 
Corollary \ref{corollary_LR_diff_l}, but we can still take 
$\lambda \in (\diff_k,\diff_{k+1})$ and solve ~\eqref{SkS_QUBO_Lagrangian_1} and verify, what happens.

For the bipartite graphs and the D-Wave graphs, we have computed stability number $\alpha=\alpha(G)$,  $\diff_{\alpha+1}$, $\diff_{\alpha+2}$ and the ~\eqref{SkS_QUBO_Lagrangian_1} using the SCIP solver \cite{SCIP}.

For bipartite graphs we observed that $\diff_{\alpha+1}<\diff_{\alpha+2}$,
% hence we took
% $\lambda = \frac{\diff_{\alpha+1}+\diff_{\alpha+2}}{2}$, 
while for some of the considered D-Wave graphs, we observed the opposite situation.
% , so we took $\lambda = \diff_{\alpha+1}+0.1$.

%, since any 
%$0\le \lambda < \diff_{\alpha+1} = m_{\alpha+1}-m_{\alpha} \le \hat d$, where $\hat d$ is the degree of the added vertex to the  maximum stable set.  Proposition \ref{roman:obs}
%implies that the maximum degree in the sparsest $(\alpha +1)$-subgraph has degree strictly less than the degree of the added vertex, hence 0 optimum solution of ~\eqref{SkS_QUBO_Lagrangian_1} would be a maximum stable set

%MENTION THAT MAXIIMUM $\alpha+1$ subraph IS MAXIMUM SATBEL SET PLUS VERTEX WITH SMALLEST DEGREE.
%what happens 

%\textcolor{red}{revise:
%The condition of monotonically increasing $\{\diff_\ell\}$ is not met for arbitrary graphs. However, there are specific scenarios where this property holds, allowing for the determination of an exact $\lambda$ that yields the desired $k$. For instance, when $k = \alpha + 1$ where $\alpha$ is the maximum stability number, the $\{\diff_\ell\}$ sequence exhibits monotonic increasing behavior. In such cases, if $\lambda$ is chosen within the appropriate range defined by $\diff_k$ and $\diff_{k+1}$, an exact solver can find the optimal sparsest $k$-subgraph in a single iteration.

%To illustrate this, consider graphs where the sparsest $(\alpha(G)+1)$-subgraph and $(\alpha(G)+2)$-subgraph are known. If the condition $\diff_{\alpha(G)+1} < \diff_{\alpha(G)+2}$ holds, then any $\lambda$ between these two values would theoretically lead to the exact solution for $k = \alpha(G)+1$ or $k = \alpha(G)+2$ when using an exact solver for ~\eqref{SkS_QUBO_Lagrangian}.
%}
\begin{table}[]
    \centering
    \resizebox{\linewidth}{!}{%
\begin{tabular}{l|cccc|cc|cc}
\hline
% Instance & $\alpha$ & Target $k$ & $\text{diff}_{\alpha+1}$ & $\text{diff}_{\alpha+2}$ & $\lambda$ & $k_{computed}$ & edges \\

\multirow{2}{*}{\textbf{Instance}} & 
\multirow{2}{*}{$\boldsymbol{\alpha}$} &
\multirow{2}{*}{$\boldsymbol{k}$} & 
\multirow{2}{*}{$\boldsymbol{\text{diff}_{\alpha+1}}$} & 
\multirow{2}{*}{$\boldsymbol{\text{diff}_{\alpha+2}}$} & 
\multicolumn{2}{c|}{$\boldsymbol{\lambda = \text{diff}_{\alpha+1}-\epsilon}$} &
\multicolumn{2}{c}{$\boldsymbol{\lambda = \text{diff}_{\alpha+1}+\epsilon}$} \\
\cline{6-9}
 &  &  &  &  & $k_{computed}$ & edges & $k_{computed}$ & edges \\

\hline
$BG_{30,30,0.8}$ & 30 & 31 & 16 & 19 & 30 & 0 & \textbf{31} & 16 \\
$BG_{30,40,0.8}$ & 40 & 41 & 27 & 29 & 40 & 0 & \textbf{41} & 27 \\
$BG_{30,50,0.8}$ & 50 & 51 & 35 & 36 & 50 & 0 & \textbf{51} & 35 \\
$BG_{30,60,0.8}$ & 60 & 61 & 44 & 45 & 60 & 0 & \textbf{61} & 44 \\
$BG_{30,70,0.8}$ & 70 & 71 &  51 & 52 & 70 & 0 & \textbf{71} & 51 \\
$DW_{50}$ & 15 & 16 & 1 & 1 & 15 & 0 & 19 & 4 \\
$DW_{75}$ & 25 & 26 & 4 & 1 & 34 & 16 & 34 & 16 \\
$DW_{100}$ & 32 & 33 & 1 & 1 & 32 & 0 & 36 & 4 \\
$DW_{200}$ & 63 & 64 & 3 & 1 & 83 & 34 & 83 & 34 \\
$DW_{300}$ & 95 & 96 & 2 & 1 & 116 & 31 & 124 & 47 \\
\hline
\end{tabular}
    }
    \caption{Results from the exact solver for the Lagrangian Relaxation \eqref{SkS_QUBO_Lagrangian_1} for selected instances, $\epsilon = 0.1$.}
    \label{tab: LR_exact}
\end{table}

Table~\ref{tab: LR_exact} contains the results obtained by solving ~\eqref{SkS_QUBO_Lagrangian_1} with exact solver BiqBin, for selected benchmark instances. For the cases where $\diff_{\alpha+1} < \diff_{\alpha+2}$, choosing a $\lambda$ in the interval $(\text{diff}_{\alpha+1}, \text{diff}_{\alpha+2})$ consistently yields a solution of cardinality $\alpha + 1$, hence optimal solutions, as desired. These instances serve as evidence that local monotonicity, i.e. $\diff_k <\diff_{k+1}$ for particular $k$ might be enough to assure that an optimal solution of ~\eqref{SkS_QUBO_Lagrangian_1} is also optimal  for \ref{SkS_problem_standard_formulation}, for given $k$.

In contrast, for the  DW graphs, the behavior is less predictable. Specifically, on instances where $\diff_{\alpha+1} > \diff_{\alpha+2}$ ($DW_{75}$, $DW_{200}$, $DW_{300}$), even the local monotonicity property  is violated. In such cases, setting $\lambda = \text{diff}_{\alpha+1} + \varepsilon$ fails to produce a subgraph of the desired size. In these cases, for any value of $\lambda$, the optimum solution of ~\eqref{SkS_QUBO_Lagrangian_1} will not be of the desired cardinality $k$, as impled by Corollary \ref{cor: diff_k>diff_k_plus_1}.

%In these cases, if $\lambda < \diff_{\alpha+1}$, the exact solver returns a maximum independent set
%For these cases, \ref{SkS_QUBO_Lagrangian} will not give the taget $k$ for any valur of $\lambda$.

% \textcolor{red}{Do we have these results: 
%For the cases where $\diff_{\alpha+1} = \diff_{\alpha+2}$, setting $\lambda = \text{diff}_{\alpha+1} + \varepsilon$, the exact solver yields a sparsest subgraph with maximum possible cardinality under the relaxed formulation. In these cases, the required cardinality can be obtained by removing the additional nodes greedily until target $k$ is reached, as outlined by the third statement of Proposition \ref{roman:obs}.

%\textcolor{red}{Omkar, this sentence suggests that we apply Prop 9 directly. But prop 9 needs that $\lambda = d_S(v)$, which is not true in our case. }
%\Omkar{Yes, you are right. The statement only holds if $\lambda \in \mathrm{Z}$, which is not the case. We can remove the para.}% }

\subsubsection{Augmented Lagrangian relaxation}

For the same instances as in Table \ref{tab: QP_biqbin}, we solve the augmented Lagrangian relaxation ~\eqref{SkS_QUBO_formulation_two} with the BiqBin solver. We consider the values $\lambda = \frac{1}{2}(k-1)$ and $\mu = k$, as suggested by Lemma \ref{lemma_ALM_SkS}. The computations are shown in Table \ref{tab: AugLag_biqbin}.
The last two columns of this table demonstrate that the cardinalities of the computed solutions are always equal to the true values of $k$, and that the numbers of edges in the subgraphs induced by the computed solutions of ~\eqref{SkS_QUBO_formulation_two} are always equal to the optimal values of the original problem, as guaranteed by Lemma \ref{lemma_ALM_SkS}.

\begin{table}[H]
    \centering
    \resizebox{\linewidth}{!}{%
\begin{tabular}{lcccccccc}
\hline
 \textbf{Instance} & $\boldsymbol{k}$ & \textbf{Optimum} & $\boldsymbol{\lambda}$ & $\boldsymbol{\mu}$ & \textbf{time} & \textbf{B\&B nodes} & $\boldsymbol{k_{computed}}$ & \textbf{edges} \\
\hline
kcluster40\_025\_10\_1 & 10 & 29 & 4.50 & 10.00 & 0.303 & 1 & 10 & \textbf{29} \\
kcluster40\_050\_10\_1 & 10 & 40 & 4.50 & 10.00 & 0.502 & 1 & 10 & \textbf{40} \\
kcluster40\_075\_10\_1 & 10 & 45 & 4.50 & 10.00 & 16.590 & 1903 & 10 & \textbf{45} \\
kcluster80\_025\_20\_1 & 20 & 94 & 9.50 & 20.00 & 12.598 & 9 & 20 & \textbf{94} \\
kcluster80\_050\_20\_1 & 20 & 149 & 9.50 & 20.00 & 2.750 & 1 & 20 & \textbf{149} \\
kcluster80\_075\_20\_1 & 20 & 182 & 9.50 & 20.00 & 92.792 & 1237 & 20 & \textbf{182} \\
kcluster100\_025\_25\_1 & 25 & 139 & 12.00 & 25.00 & 18.758 & 21 & 25 & \textbf{139} \\
kcluster100\_050\_25\_1 & 25 & 218 & 12.00 & 25.00 & 19.284 & 47 & 25 & \textbf{218} \\
kcluster100\_075\_25\_1 & 25 & 282 & 12.00 & 25.00 & 164.038 & 1357 & 25 & \textbf{282} \\
\bottomrule
\end{tabular}
    }
    \caption{Results of  BiqBin solver for $ER$ graphs using Augmented Lagrangian Relaxation \eqref{SkS_QUBO_formulation_two}.}
    \label{tab: AugLag_biqbin}
\end{table}

% }

% \subsubsection{Lagrangian Relaxation Method}
% \begin{algorithm}
% \caption{Lagrangian relaxation for the SkS problem}
% \label{alg:lagrangian_relaxation}
% \begin{algorithmic}[1]
% \Require Value of $k$, initial $\lambda$
% \Ensure Final $\lambda$ for use in the Lagrangian relaxation
% \Repeat
%     \State $x \gets \arg \min_{x \in \{0, 1\}^n} \frac{1}{2}x^TAx + \lambda(k - e^Tx)$
%     \If{$k - e^Tx \neq 0$}
%         \State $\lambda \gets \lambda + \varphi (k - e^Tx)$
%         \Comment{$\varphi$ is the stepsize at this step}
%     \EndIf
% \Until{$k - e^Tx = 0$ or iteration limit reached}
% \end{algorithmic}
% \end{algorithm}

% Note that the efficiency of Algorithm~\ref{alg:lagrangian_relaxation} depends on the initial value of the Lagrangian multiplier as well as the choice of the step size. Both factors play a crucial role in determining the convergence speed and the quality of the solution.

% \OB[Similar to Algorithm \ref{alg:lagrangian_relaxation}, the choice of the step size, $\rho$ matters, which I think has dependence on the structure of the graph. I don't know how we will compare the 3 methods, i.e. LR, ALM and the proposed algorithm.]
% $\vdots$

\subsection{Iterative Methods}
%\textcolor{red}{Omkar, add a paragraph with motivation - why we developed the algorithm on this section? Partial answer is that we do not want to use all the time exact solver since it is time consuming and we do not always  know good  values for penalty parameters.}

Solving ~\eqref{SkS_problem_standard_formulation} and its relaxations ~\eqref{SkS_QUBO_formulation_one}, ~\eqref{SkS_QUBO_Lagrangian_1}, and ~\eqref{SkS_QUBO_formulation_two} exactly is, in general, computationally expensive due to the combinatorial explosion of the search space as the graph size increases. Exact solvers quickly become impractical for larger instances. Moreover, the Lagrange multiplier $\lambda$ and quadratic penalty parameter $\mu$, which are critical for penalizing constraint violation, are known in advance only in specific cases, outlined in Section \ref{definitions}, and vary significantly between instances. However, if we use approximate algorithms instead of exact solvers, then the theoretical results about these parameters from Section \ref{definitions} are no longer relevant.

These limitations motivate the development of adaptive, iterative methods that can efficiently explore the solution space without requiring exact parameter tuning and without the need for exact solvers. The iterative algorithms introduced in this section dynamically update penalty parameters based on the intermediate solution quality, enabling gradual enforcement of the cardinality constraint, and employ algorithms that solve QUBO problems approximately.

In the rest of this subsection, we propose and implement three iterative algorithms that balance objective optimization with constraint satisfaction. Each method progressively adjusts the penalty or Lagrangian parameters to guide the solution toward a subgraph of the desired cardinality $k$. In this study, we use two heuristic solvers:  the Simulated Annealing (SA) solver and the D-Wave quantum processing unit (QPU) solver. 

SA is a probabilistic metaheuristic inspired by the annealing process in metallurgy. It explores the solution space by iteratively moving the current solution to a neighboring one. Moves that improve the objective function are always accepted, whereas moves that worsen it are accepted with a probability that decreases over time. This allows the algorithm to escape the local optima. For our experiments, we used the SA implementation from \texttt{dwave-ocean-sdk}, which is a versatile tool to solve QUBO problems.

We used the D-Wave Advantage2 quantum annealer, which is designed to solve QUBO problems by mapping them to its underlying qubit architecture. The QPU leverages quantum effects such as superposition and tunneling to find low-energy states of the physical system, which correspond to optimal or near-optimal solutions of the QUBO problem.  For problems that do not match the native hardware graph (Zephyr) graph, an embedding process is required, which maps the logical variables of the problem to physical qubits. 

The computations in this section are less intensive compared to the exact approaches discussed previously. Consequently, we considered the full set of 270 ER instances. However, due to space limitations, the tables in the following subsections report results for only 15 representative instances. Specifically, for each $(n, p)$ pair, we selected the first instance within the subgroup corresponding to $k = n/4$. A comprehensive summary based on the full ER dataset is provided in the Discussion section.

% The former was run locally, while the latter was run on the Advantage2 system at Forschungszentrum Jülich. Below, we describe each method and its application to benchmark instances. 

%\textcolor{red}{Omkar, please  more explanation about SA and QPU.} \Omkar{Done}

\subsubsection{Quadratic Penalty Iterative Algorithm (QPIA)}
In this method, we consider ~\eqref{SkS_QUBO_formulation_one} and try solve it iteratively, where in each iteration we solve ~\eqref{SkS_QUBO_formulation_one}  approximately and update the quadratic penalty parameter. 
To ensure sufficient enforcement of the constraint, we initialize the quadratic penalty parameter using $\mu = 2\min(\tilde{m}, k-1) + 1$, where $\tilde{m}$ is the number of edges in the subgraph found by greedy heuristic. This choice follows ~\eqref{choose_quadratic_penalty}, and ensures that the penalty is strong enough to discourage infeasible solutions. The maximum number of iterations was set to 100.

The QPIA algorithm is outlined as Algorithm \ref{alg: QPIA_algorithm}.

\begin{algorithm}[H]
    \caption{Quadratic Penalty Iterative Algorithm  (QPIA) for the ~\eqref{SkS_QUBO_formulation_one} relaxation}
    \label{alg: QPIA_algorithm}
    \begin{algorithmic}[1]
        \Require Graph $G = (V, E)$ with $|V| = n$, target cardinality $k \le n$, initial penalty parameter $\mu_{init}$
        % , maximum iterations $T$
        \Ensure Subgraph $H \subseteq G$ with $|V(H)| = k$
        \State Initialize: $\mu \gets \mu_{init}$, $t \gets 0$
        \Repeat
            \State $x \approx \arg\min_{x \in \{0,1\}^n} \frac{1}{2} x^T A x + \frac{\mu}{2}(k - e^T x)^2$
            \Comment{Solved using a heuristic method}
            \If{$e^T x \ne k$}
                \State $\mu \gets \mu + 1$
                % \State $t \gets t + 1$
            \Else
                \State \textbf{break}
            \EndIf
        \Until{
        $e^T x = k$
        % \textbf{or} $t \ge T$
        }
        \State \Return Induced subgraph $H$ corresponding to $x$
    \end{algorithmic}
\end{algorithm}

\begin{landscape}
\begin{table}[]
    \centering
    \resizebox{\linewidth}{!}{%
\begin{tabular}{l|cccc|ccccc|ccccc}
\hline

\multirow{2}{*}{\textbf{Instance}} & 
\multirow{2}{*}{$\boldsymbol{k}$} & 
\multirow{2}{*}{\textbf{Optimum}} & 
\multirow{2}{*}{\textbf{Greedy Sol.}} & 
\multirow{2}{*}{$\boldsymbol{\mu_{init}}$} & 
\multicolumn{5}{c|}{\textbf{Simulated Annealing}} & 
\multicolumn{5}{c}{\textbf{D-Wave QPU}} \\
\cline{6-15}
& & & & &
$k_{first\_iter}$ & Iters. & $\mu$ & $best_{k_{reached}}$ & QPIA sol & $k_{first\_iter}$ & Iters. & $\mu$ & $best_{k_{reached}}$ & QPIA sol \\
\hline
kcluster40\_025\_20\_1 & 20 & 77 & \textbf{77} & 39 & 20 & 1 & 39 & 20 & 70 & 21 & 11 & 49 & 20 & 54 \\
kcluster40\_050\_20\_1 & 20 & 130 & 129 & 39 & 20 & 1 & 39 & 20 & 121 & 21 & 100 & 138 & 21 & - \\
kcluster40\_075\_20\_1 & 20 & 168 & 167 & 39 & 20 & 1 & 39 & 20 & 159 & 21 & 74 & 112 & 20 & 136 \\
kcluster80\_025\_40\_1 & 40 & 280 & 277 & 79 & 40 & 1 & 79 & 40 & 234 & 42 & 100 & 178 & 42 & - \\
kcluster80\_050\_40\_1 & 40 & 488 & 487 & 79 & 40 & 1 & 79 & 40 & 448 & 44 & 100 & 178 & 42 & - \\
kcluster80\_075\_40\_1 & 40 & 671 & 664 & 79 & 40 & 1 & 79 & 40 & 630 & 42 & 100 & 178 & 41 & - \\
kcluster100\_025\_50\_1 & 50 & 417 & 412 & 99 & 50 & 1 & 99 & 50 & 360 & 55 & 100 & 198 & 55 & - \\
kcluster100\_050\_50\_1 & 50 & 729 & 724 & 99 & 50 & 1 & 99 & 50 & 669 & 55 & 59 & 157 & 50 & 632 \\
kcluster100\_075\_50\_1 & 50 & 1029 & 1025 & 99 & 50 & 1 & 99 & 50 & 958 & 55 & 22 & 120 & 50 & 910 \\
kcluster120\_025\_60\_1 & 60 & 600 & 592 & 119 & 60 & 1 & 119 & 60 & 510 & - & - & - & - & - \\
kcluster120\_050\_60\_1 & 60 & 1060 & 1047 & 119 & 60 & 1 & 119 & 60 & 974 & - & - & - & - & - \\
kcluster120\_075\_60\_1 & 60 & 1478 & 1465 & 119 & 60 & 1 & 119 & 60 & 1384 & - & - & - & - & - \\
kcluster140\_025\_70\_1 & 70 & 806 & 798 & 139 & 70 & 1 & 139 & 70 & 684 & - & - & - & - & - \\
kcluster140\_050\_70\_1 & 70 & 1424 & 1420 & 139 & 70 & 1 & 139 & 70 & 1302 & - & - & - & - & - \\
kcluster140\_075\_70\_1 & 70 & 2004 & 1990 & 139 & 70 & 1 & 139 & 70 & 1892 & - & - & - & - & - \\
kcluster160\_025\_80\_1 & 80 & 1047 & 1038 & 159 & 80 & 1 & 159 & 80 & 894 & - & - & - & - & - \\
kcluster160\_050\_80\_1 & 80 & 1879 & 1874 & 159 & 80 & 1 & 159 & 80 & 1710 & - & - & - & - & - \\
kcluster160\_075\_80\_1 & 80 & 2592 & 2584 & 159 & 80 & 1 & 159 & 80 & 2453 & - & - & - & - & - \\
\hline
\end{tabular}
    }
    \caption{Results of the QPIA method on $ER$ graphs instances, where  SA and D-wave QPU were used in Step 3.}
    \label{tab: QPIA_ER_ins}
\end{table}

\begin{table}[]
    \centering
    \resizebox{\linewidth}{!}{%
\begin{tabular}{l|cccc|ccccc|ccccc}
\hline

\multirow{2}{*}{\textbf{Instance}} & 
\multirow{2}{*}{$\boldsymbol{k}$} & 
\multirow{2}{*}{\textbf{Optimum}} & 
\multirow{2}{*}{\textbf{Greedy Sol.}} & 
\multirow{2}{*}{$\boldsymbol{\mu_{init}}$} & 
\multicolumn{5}{c|}{\textbf{Simulated Annealing}} & 
\multicolumn{5}{c}{\textbf{D-Wave QPU}} \\
\cline{6-15}
& & & & &
$k_{first\_iter}$ & Iters. & $\mu$ & $best_{k_{reached}}$ & QPIA sol & $k_{first\_iter}$ & Iters. & $\mu$ & $best_{k_{reached}}$ & QPIA sol \\
\hline
$DW_{50}$ & 16 & 1 & 4 & 9 & 16 & 1 & 9 & 16 & 2 & 18 & 100 & 108 & 18 & - \\
$DW_{75}$ & 26 & 4 & 7 & 15 & 26 & 1 & 15 & 26 & 7 & 37 & 100 & 114 & 35 & - \\
$DW_{100}$ & 33 & 1 & 7 & 15 & 33 & 1 & 15 & 33 & 13 & 43 & 100 & 114 & 42 & - \\
$DW_{200}$ & 64 & 3 & 14 & 29 & 64 & 1 & 29 & 64 & 85 & - & - & - & - & - \\
$DW_{300}$ & 96 & 2 & 23 & 47 & 96 & 1 & 47 & 96 & 154 & - & - & - & - & - \\
$DW_{400}$ & 124 & 3 & 39 & 79 & 124 & 1 & 79 & 124 & 225 & - & - & - & - & - \\
$DW_{500}$ & 152 & 1 & 43 & 87 & 152 & 1 & 87 & 152 & 280 & - & - & - & - & - \\
\hline
\end{tabular}
    }
    \caption{Results of the QPIA method on $DW$ graphs instances, where  SA and D-wave QPU were used in Step 3.}
    \label{tab: QPIA_DW_ins}
\end{table}

\end{landscape}

Table~\ref{tab: QPIA_ER_ins} displays performance metrics of Algorithm \ref{alg: QPIA_algorithm} across selected  ER  instances. Similarly. Table \ref{tab: QPIA_DW_ins} shows the performance for  DW graphs. Alongside the known optimum and greedy baseline, the table lists key indicators such as the initial $\mu$, the solution size at the first iteration, the number of penalty updates, and the final solution quality. 

We note that for D-Wave QPU, since the quadratic penalty term introduces a fully connected interaction among all qubits, the resulting QUBO represents a dense (effectively complete) graph. This necessitates the use of clique embedding for mapping the problem to the Zephyr architecture. However, cliques of size greater than approximately 100 cannot be embedded due to hardware limitations. As a result, some large instances could not be solved using QPU, and their entries are left blank in the table.

%\textcolor{red}{Omkar, please explain better what does it mean: Since non-zero $\mu$ parameter indicates a complete graph, we use clique embedding to embed the problem in Zephyr processor. The cliques with size \(> 100\) are not embeddable on the QPU. 
%} \Omkar{Done.}

For the  ER  cases, we observe that for some small instances (with $n \leq 100$), feasible solutions were still not obtained with the D-Wave QPU solver within the maximum iteration limit ($best_{k_{reached}}\neq k$). Also, we were not able to find feasible solutions for the  DW graphs, and the corresponding cells in the column QPIA are left empty.

On the other hand, the SA solver always returns a feasible solution within the first iteration, though the quality is often inferior to the greedy baseline. This may be attributed to the relatively high value of $\mu_{init}$, which makes the problem difficult to optimize heuristically.  

Note that in the  ER  cases, we solve the SkS on the complement of the original graph and subsequently complement the solution again to get the densest $k$-subgraph, as explained in \ref{sec:Data}. Consequently, here the higher value of the solution is considered better, unlike in the case of  DW instances, where we compute and report SkS.

In summary, this iterative algorithm performs very weakly.

%Next, to obtain the better solution, we consider the approaches in which we boost the diagonal entries of the QUBO, i.e. considering $|\lambda| > |\mu|$.

%\textcolor{red}{More comments: this approach is actually very weak. }

\subsubsection{Lagrangian Relaxation Iterative Algorithm (LRIA)}

% Here, we consider an approach where we do not have quadratic penalty term. We just consider the linear Lagrange parameter, $\lambda$. Starting with $\lambda_{init}$, the algorithm computes the cardinality of the the solution of equation ~\eqref{SkS_QUBO_Lagrangian_1}. Based on the deviation from the target cardinality, The algorithm dynamically tunes $\lambda$, optionally refining the best solution via a greedy approach if the solution with required cardinality is not found. 
In this approach, we solve ~\eqref{SkS_problem_standard_formulation} by considering its Lagrangian relaxation ~\eqref{SkS_QUBO_Lagrangian_1}, where we have a linear penalty term with Lagrange multiplier $\lambda$ to enforce the cardinality constraint. 

The multiplier $\lambda$ is then updated iteratively based on the deviation of the computed solution size from the target cardinality $k$. The update follows a proportional adjustment using a fixed step size $\varphi$, which we set to $0.1$. 

Throughout the process, the algorithm tracks $x_{best}$, the solution with the fewest edges found so far that has a cardinality $k_{best}$ greater than or equal to the target $k$. If the target cardinality is not reached within the maximum iterations (which we set to 100), a greedy refinement step is applied to $x_{best}$: vertices with the highest degree within the subgraph are iteratively removed until the target cardinality $k$ is reached. The algorithm returns no solution if a subgraph with at least $k$ vertices is never found (i.e., if $k_{best} < k$). The full procedure is outlined in Algorithm \ref{alg: LRIA_algorithm}.

We initialize the Lagrange multiplier with $\lambda_{init} = \Delta(S_{\text{greedy}})$, where $\Delta(S_{\text{greedy}})$ denotes the maximum degree of the subgraph found using a greedy heuristic. This choice is motivated by the second statement  of  Proposition \ref{roman:obs}. 

\begin{algorithm}[H]
    \caption{Lagrangian Relaxation Iterative Algorithm for  ~\eqref{SkS_QUBO_Lagrangian_1} relaxation}
    \label{alg: LRIA_algorithm}
    \begin{algorithmic}[1]
        \Require Graph $G = (V, E)$ with $|V| = n$, target cardinality $k \le n$, $\lambda_{init}$, step size $\varphi$
        \Ensure Subgraph $H \subseteq G$ with $|V(H)| = k$
        \State Initialize: $\lambda \gets \lambda_{init}$, $t \gets 0$
        \Repeat
            \State $x \approx \arg\min_{x \in \{0,1\}^n} \frac{1}{2} x^T A x - \lambda e^T x$
            \Comment{Solved using a heuristic}
            \State $k_{\text{computed}} \gets e^T x$
            \If{($k_{\text{computed}} \ge k$ \textbf{and} $k_{\text{computed}} < k_{\text{best}}$)}
                \State Update best solution: $x_{\text{best}} \gets x$, $k_{\text{best}} \gets k_{\text{computed}}$
            \EndIf
            \If{$k_{\text{computed}} \ne k$}
                \State $\lambda \gets \lambda + \varphi \cdot (k - k_{\text{computed}})$
            \Else
                \State \textbf{break}
            \EndIf
        \Until{$k_{\text{computed}} = k$}
        \If{$k_{\text{best}} \ne k$}
            \State Apply greedy refinement to $x_{\text{best}}$ to obtain $x_{\text{greedy}}$ with $k$ elements
        \EndIf
        \State \Return Induced subgraph $H$ corresponding to $x_{\text{best}}$ or $x_{\text{greedy}}$
    \end{algorithmic}
\end{algorithm}

\begin{landscape}
\begin{table}[]
    \centering
    \resizebox{\linewidth}{!}{%
\begin{tabular}{l|cccc|ccccc|ccccc}
\hline
\multirow{2}{*}{\textbf{Instance}} & 
\multirow{2}{*}{$\boldsymbol{k}$} & 
\multirow{2}{*}{\textbf{Optimum}} & 
\multirow{2}{*}{\textbf{Greedy Sol.}} & 
\multirow{2}{*}{$\boldsymbol{\lambda_{init}}$} & 
\multicolumn{5}{c|}{\textbf{Simulated Annealing}} & 
\multicolumn{5}{c}{\textbf{D-Wave QPU}} \\
\cline{6-15}
& & & & &
$k_{first\_iter}$ & Iters. & $\lambda$ & $best_{k_{reached}}$ & LRIA sol & $k_{first\_iter}$ & Iters. & $\lambda$ & $best_{k_{reached}}$ & LRIA sol \\
\hline

kcluster40\_025\_20\_1 & 20 & 77 & \textbf{77} & 14 & 19 & 3 & 14.00 & 20 & \textbf{77} & 20 & 1 & 14.00 & 20 & 62 \\
kcluster40\_050\_20\_1 & 20 & 130 & 129 & 9 & 20 & 1 & 9.00 & 20 & \textbf{130} & 20 & 1 & 9.00 & 20 & \textbf{130} \\
kcluster40\_075\_20\_1 & 20 & 168 & 167 & 4 & 18 & 2 & 4.20 & 20 & \textbf{168} & 19 & 2 & 4.10 & 20 & \textbf{168} \\
kcluster80\_025\_40\_1 & 40 & 280 & 277 & 29 & 40 & 1 & 29.00 & 40 & \textbf{280} & 37 & 23 & 32.00 & 40 & 196 \\
kcluster80\_050\_40\_1 & 40 & 488 & 487 & 18 & 40 & 1 & 18.00 & 40 & \textbf{488} & 38 & 3 & 18.30 & 40 & 435 \\
kcluster80\_075\_40\_1 & 40 & 671 & 664 & 8 & 38 & 2 & 8.20 & 40 & \textbf{671} & 39 & 4 & 8.30 & 40 & 670 \\
kcluster100\_025\_50\_1 & 50 & 417 & 412 & 36 & 49 & 3 & 36.00 & 50 & \textbf{417} & 48 & 2 & 36.20 & 50 & 327 \\
kcluster100\_050\_50\_1 & 50 & 729 & 724 & 25 & 52 & 100 & 23.80 & 52 & \textbf{729} & 52 & 11 & 23.10 & 50 & 658 \\
kcluster100\_075\_50\_1 & 50 & 1029 & 1025 & 11 & 50 & 1 & 11.00 & 50 & \textbf{1029} & 51 & 2 & 10.90 & 50 & 1027 \\
kcluster120\_025\_60\_1 & 60 & 600 & 592 & 43 & 59 & 2 & 43.10 & 60 & \textbf{600} & 62 & 7 & 42.20 & 60 & 479 \\
kcluster120\_050\_60\_1 & 60 & 1060 & 1047 & 28 & 60 & 1 & 28.00 & 60 & \textbf{1060} & 58 & 5 & 28.70 & 60 & 967 \\
kcluster120\_075\_60\_1 & 60 & 1478 & 1465 & 13 & 57 & 10 & 14.10 & 60 & \textbf{1478} & 53 & 9 & 14.30 & 60 & 1428 \\
kcluster140\_025\_70\_1 & 70 & 806 & 798 & 50 & 69 & 2 & 50.10 & 70 & \textbf{806} & 69 & 4 & 50.10 & 70 & 617 \\
kcluster140\_050\_70\_1 & 70 & 1424 & 1420 & 33 & 70 & 1 & 33.00 & 70 & \textbf{1424} & 70 & 1 & 33.00 & 70 & 1258 \\
kcluster140\_075\_70\_1 & 70 & 2004 & 1990 & 16 & 70 & 1 & 16.00 & 70 & \textbf{2004} & 61 & 11 & 18.30 & 70 & 1904 \\
kcluster160\_025\_80\_1 & 80 & 1047 & 1038 & 56 & 78 & 16 & 58.00 & 80 & \textbf{1047} & - & 1 & 56.00 & - & - \\
kcluster160\_050\_80\_1 & 80 & 1879 & 1874 & 38 & 81 & 2 & 37.90 & 80 & \textbf{1879} & - & 1 & 38.00 & - & - \\
kcluster160\_075\_80\_1 & 80 & 2592 & 2584 & 17 & 76 & 100 & 18.30 & 82 & \textbf{2592} & 75 & 9 & 17.80 & 80 & 2470 \\

\hline
\end{tabular}
    }
    \caption{Results of the LRIA method on $ER$ graphs instances, where SA and D-wave QPU solvers are used in Step 3.}
    \label{tab: LRIA_ER_ins}
\end{table}
    
\begin{table}[]
    \centering
    \resizebox{\linewidth}{!}{%
\begin{tabular}{l|cccc|ccccc|ccccc}
\hline
\multirow{2}{*}{\textbf{Instance}} & 
\multirow{2}{*}{$\boldsymbol{k}$} & 
\multirow{2}{*}{\textbf{Optimum}} & 
\multirow{2}{*}{\textbf{Greedy Sol.}} & 
\multirow{2}{*}{$\boldsymbol{\lambda_{init}}$} & 
\multicolumn{5}{c|}{\textbf{Simulated Annealing}} & 
\multicolumn{5}{c}{\textbf{D-Wave QPU}} \\
\cline{6-15}
& & & & &
$k_{first\_iter}$ & Iters. & $\lambda$ & $best_{k_{reached}}$ & LRIA sol & $k_{first\_iter}$ & Iters. & $\lambda$ & $best_{k_{reached}}$ & LRIA sol \\
\hline
$DW_{50}$ & 16 & 1 & 4 & 1 & 15 & 9 & 1.00 & 16 & \textbf{1} & 24 & 98 & 0.60 & 16 & 4 \\
$DW_{75}$ & 26 & 4 & 7 & 1 & 24 & 100 & 1.30 & 28 & \textbf{4} & 34 & 100 & 0.50 & 27 & 8 \\
$DW_{100}$ & 33 & 1 & 7 & 1 & 34 & 3 & 1.00 & 33 & \textbf{1} & 46 & 100 & 0.90 & 34 & 7 \\
$DW_{200}$ & 64 & 3 & 14 & 1 & 62 & 27 & 1.00 & 64 & \textbf{3} & 83 & 100 & -0.20 & 71 & 16 \\
$DW_{300}$ & 96 & 2 & 23 & 1 & 95 & 65 & 1.00 & 96 & 6 & 126 & 100 & -0.10 & 100 & 22 \\
$DW_{400}$ & 124 & 3 & 39 & 1 & 120 & 100 & 10.60 & 147 & 17 & 162 & 100 & -2.40 & 162 & 29 \\
$DW_{500}$ & 152 & 1 & 43 & 1 & 154 & 100 & -4.20 & 154 & 14 & 194 & 100 & -5.50 & 194 & 43 \\
\hline
\end{tabular}
    }
    \caption{Results of the LRIA method on $DW$ graphs instances, where SA and D-wave QPU solvers are used in Step 3.}
    \label{tab: LRIA_DW_ins}
\end{table}

\end{landscape}

Table \ref{tab: LRIA_ER_ins} summarizes the results obtained by Algorithm \ref{alg: LRIA_algorithm} on benchmark instances using both heuristics - SA and QPU. For each instance, the table reports the known optimum, the greedy solution value, and the $\lambda_{init}$ we used. The table also shows the cardinality found in the first iteration ($k_{first\_iter}$), the number of iterations until termination, the final value of $\lambda$, the best cardinality reached ($best_{k_{reached}}$), and the final solution value returned by the algorithm.

For the  ER  graphs, we observe that in all reported cases,  the optimum solution was obtained by using the SA solver. In a couple of cases where $best_{k_{reached}}$ was not equal to target $k$, the algorithm returns the solution with the required cardinality by removing extra nodes from the best found subgraph greedily (in the complement graph, where we try to compute the sparsest $k$-subgraphs, we iteratively remove vertices with the largest degree in the induced subgraph).  Additionally, the initial and final values of $\lambda$ are very close to each other, if not the same. This informs us that the value of $\lambda_{init} = \Delta(S_{greedy})$ is a reasonable starting point.
%\textcolor{red}{Omkar, these details, like how we select $\lambda_{init}$, must be explained at the beginning of this subsection, before the algorithm is introduced.}\Omkar{Done.}

We also observe that using the D-Wave QPU solver, we get a better solution by using the LRIA method compared to the earlier QPIA method. For two small instances, we even reach the optimum solution. Additionally, since in this approach, we don't have a quadratic penalty factor, we do not need to embed the graph using clique embedding. This allows us to embed bigger graphs on the QPU. In a few instances where table entries are missing, we were unable to find a suitable embedding.

Similarly, for the  DW graphs, the LRIA solution using the SA solver is better than the greedy baseline in all the cases. Moreover, the QPU solution is at par or better than the greedy baseline in 5 of the 7 cases, suggesting that for the problems which can be embedded directly on the QPU, the solution quality is better though not optimum.

Note that in the cases where the final value of $\lambda$ is negative, the solution is obtained by greedily removing nodes from the graph corresponding to $best_{k_{reached}}$, as a negative value of $\lambda$ corresponds to the solution with all zeros, i.e., the graph with no nodes and edges.

\subsubsection{Augmented Lagrangian Iterative Algorithm (ALIA)}

The Augmented Lagrangian Iterative Algorithm (ALIA) is designed to solve the relaxation ~\eqref{SkS_QUBO_formulation_two}, which augments the objective with both a Lagrangian term and a quadratic penalty.

We initialize the Lagrange multiplier as $\lambda_{init} = \Delta(S_{greedy})$, and set the initial penalty parameter, $\mu_{init} = 0.1$. At each iteration, we update $\lambda$ based on the constraint violation and slightly increase $\mu$ by a factor of $\rho = 1.1$. The process continues until the size constraint is satisfied or the iteration limit (set to 100) is reached.

\begin{algorithm}
    \caption{Augmented Lagrangian Iterative Algorithm for ~\eqref{SkS_problem_standard_formulation}}
    \label{alg: ALIA_algorithm}
    \begin{algorithmic}[1]
        \Require Value of $k$, initial $\lambda$, $\mu$, and increase factor $\rho > 1$
        \Ensure Final $\lambda$ and $\mu$ for use in the augmented Lagrangian function
        \Repeat
            \State $x \approx \arg \min_{x \in \{0, 1\}^n} \frac{1}{2}x^TAx + \lambda(k - e^Tx) + \frac{\mu}{2}(k - e^Tx)^2 $
            \Comment{Solved using a heuristic}
            \If{$k - e^Tx \neq 0$}
                \State $\lambda \gets \lambda + \mu (k - e^Tx)$
            \EndIf
            \State $\mu \gets \mu \rho$
        \Until{$k - e^Tx = 0$ or iteration limit reached}
        \State \Return Induced subgraph $H$ corresponding to $x$
    \end{algorithmic}
\end{algorithm}

The results of ALIA method for selected  ER  instances are summarized in Table \ref{tab: ALIA_ER_ins}. Each row corresponds to the specific instance and includes performance metrics under both solvers.

We observe that for all instances, the SA solver successfully returns a feasible and optimum solution. In most cases, the algorithm converges in a few iterations, indicating that the initial values of $\lambda$ and $\mu$ are reasonable and effective.

For D-Wave QPU, similar to QPIA, the presence of a non-zero quadratic penalty $\mu$ leads to a dense QUBO formulation, requiring clique embedding in the Zephyr architecture. Instances involving more than 100 variables cannot be embedded due to hardware limitations. Consequently, for large instances, the results are empty. In some cases, even for embeddable problem sizes, the algorithm fails to converge to a feasible solution within the iteration limit, resulting in missing objective values.

The results for  DW graphs are tabulated in Table \ref{tab: ALIA_DW_ins}. 
We observe that for these instances, we get a feasible solution using the SA solver. Similarly, we get a feasible solution for all the embeddable cases with the D-Wave QPU solver, unlike in the case of the QPIA.

\begin{landscape}
    \begin{table}[]
    \centering
    \resizebox{\linewidth}{!}{%
\begin{tabular}{l|cccc|cccccc|cccccc}
\hline
\multirow{2}{*}{\textbf{Instance}} & 
\multirow{2}{*}{$\boldsymbol{k}$} & 
\multirow{2}{*}{\textbf{Optimum}} & 
\multirow{2}{*}{\textbf{Greedy Sol.}} & 
\multirow{2}{*}{$\boldsymbol{\lambda_{init}}$} & 
\multicolumn{6}{c|}{\textbf{Simulated Annealing}} & 
\multicolumn{6}{c}{\textbf{D-Wave QPU}} \\
\cline{6-17}
& & & & &
$k_{first\_iter}$ & Iters. & $\lambda$ & $\mu$ & $best_{k_{reached}}$ & ALIA sol & $k_{first\_iter}$ & Iters. & $\lambda$ & $\mu$ & $best_{k_{reached}}$ &  ALIA sol \\
\hline
kcluster40\_025\_20\_1 & 20 & 77 & \textbf{77} & 14 & 20 & 1 & 14.00 & 0.10 & 20 & \textbf{77} & 19 & 2 & 14.10 & 0.11 & 20 & 57 \\
kcluster40\_050\_20\_1 & 20 & 130 & 129 & 9 & 20 & 1 & 9.00 & 0.10 & 20 & \textbf{130} & 19 & 2 & 9.10 & 0.11 & 20 & 124 \\
kcluster40\_075\_20\_1 & 20 & 168 & 167 & 4 & 20 & 1 & 4.00 & 0.10 & 20 & \textbf{168} & 20 & 1 & 4.00 & 0.10 & 20 & 167 \\
kcluster80\_025\_40\_1 & 40 & 280 & 277 & 29 & 40 & 1 & 29.00 & 0.10 & 40 & \textbf{280} & 43 & 27 & 4.51 & 1.19 & 40 & 179 \\
kcluster80\_050\_40\_1 & 40 & 488 & 487 & 18 & 40 & 1 & 18.00 & 0.10 & 40 & \textbf{488} & 44 & 26 & -2.87 & 1.08 & 40 & 393 \\
kcluster80\_075\_40\_1 & 40 & 671 & 664 & 8 & 40 & 1 & 8.00 & 0.10 & 40 & \textbf{671} & 42 & 9 & 5.68 & 0.21 & 40 & 604 \\
kcluster100\_025\_50\_1 & 50 & 417 & 412 & 36 & 50 & 1 & 36.00 & 0.10 & 50 & \textbf{417} & 55 & 100 & -9394.13 & 1252.78 & 53 & - \\
kcluster100\_050\_50\_1 & 50 & 729 & 724 & 25 & 52 & 5 & 24.07 & 0.15 & 50 & \textbf{729} & 48 & 10 & 27.50 & 0.24 & 50 & 636 \\
kcluster100\_075\_50\_1 & 50 & 1029 & 1025 & 11 & 50 & 1 & 11.00 & 0.10 & 50 & \textbf{1029} & 53 & 100 & -18655.95 & 1252.78 & 53 & - \\
kcluster120\_025\_60\_1 & 60 & 600 & 592 & 43 & 60 & 1 & 43.00 & 0.10 & 60 & \textbf{600} & - & - & - & - & - & - \\
kcluster120\_050\_60\_1 & 60 & 1060 & 1047 & 28 & 60 & 1 & 28.00 & 0.10 & 60 & \textbf{1060} & - & - & - & - & - & - \\
kcluster120\_075\_60\_1 & 60 & 1478 & 1465 & 13 & 59 & 8 & 13.95 & 0.19 & 60 & \textbf{1478} & - & - & - & - & - & - \\
kcluster140\_025\_70\_1 & 70 & 806 & 798 & 50 & 70 & 1 & 50.00 & 0.10 & 70 & \textbf{806} & - & - & - & - & - & - \\
kcluster140\_050\_70\_1 & 70 & 1424 & 1420 & 33 & 70 & 1 & 33.00 & 0.10 & 70 & \textbf{1424} & - & - & - & - & - & - \\
kcluster140\_075\_70\_1 & 70 & 2004 & 1990 & 16 & 70 & 1 & 16.00 & 0.10 & 70 & \textbf{2004} & - & - & - & - & - & - \\
kcluster160\_025\_80\_1 & 80 & 1047 & 1038 & 56 & 78 & 11 & 58.06 & 0.26 & 80 & \textbf{1047} & - & - & - & - & - & - \\
kcluster160\_050\_80\_1 & 80 & 1879 & 1874 & 38 & 80 & 1 & 38.00 & 0.10 & 80 & \textbf{1879} & - & - & - & - & - & - \\
kcluster160\_075\_80\_1 & 80 & 2592 & 2584 & 17 & 78 & 47 & 15.00 & 8.02 & 80 & \textbf{2592} & - & - & - & - & - & - \\
\hline
\end{tabular}
    }
    \caption{Results of the ALIA method on $ER$ graphs instances, where SA and D-wave QPU solvers are used in Step 2.}
    \label{tab: ALIA_ER_ins}
\end{table}

    \begin{table}[]
    \centering
    \resizebox{\linewidth}{!}{%
\begin{tabular}{l|cccc|cccccc|cccccc}
\hline
\multirow{2}{*}{\textbf{Instance}} & 
\multirow{2}{*}{$\boldsymbol{k}$} & 
\multirow{2}{*}{\textbf{Optimum}} & 
\multirow{2}{*}{\textbf{Greedy Sol.}} & 
\multirow{2}{*}{$\boldsymbol{\lambda_{init}}$} & 
\multicolumn{6}{c|}{\textbf{Simulated Annealing}} & 
\multicolumn{6}{c}{\textbf{D-Wave QPU}} \\
\cline{6-17}
& & & & &
$k_{first\_iter}$ & Iters. & $\lambda$ & $\mu$ & $best_{k_{reached}}$ & ALIA sol & $k_{first\_iter}$ & Iters. & $\lambda$ & $\mu$ & $best_{k_{reached}}$ &  ALIA sol \\
\hline
$DW_{50}$ & 16 & 1 & 4 & 1 & 16 & 1 & 1.00 & 0.10 & 16 & \textbf{1} & 19 & 4 & 0.24 & 0.13 & 16 & 4 \\
$DW_{75}$ & 26 & 4 & 7 & 1 & 25 & 32 & 0.50 & 1.92 & 26 & \textbf{4} & 24 & 2 & 1.20 & 0.11 & 26 & 8 \\
$DW_{100}$ & 33 & 1 & 7 & 1 & 33 & 1 & 1.00 & 0.10 & 33 & \textbf{1} & 30 & 7 & -0.05 & 0.18 & 33 & 52 \\
$DW_{200}$ & 64 & 3 & 14 & 1 & 64 & 1 & 1.00 & 0.10 & 64 & 5 & - & - & - & - & - & - \\
$DW_{300}$ & 96 & 2 & 23 & 1 & 96 & 1 & 1.00 & 0.10 & 96 & 5 & - & - & - & - & - & - \\
$DW_{400}$ & 124 & 3 & 39 & 1 & 124 & 1 & 1.00 & 0.10 & 124 & 13 & - & - & - & - & - & - \\
$DW_{500}$ & 152 & 1 & 43 & 1 & 152 & 1 & 1.00 & 0.10 & 152 & 10 & - & - & - & - & - & - \\
\hline
\end{tabular}
    }
    \caption{Results of the ALIA method on $DW$ graphs instances, where SA and D-wave QPU solvers are used in Step 2.}
    \label{tab: ALIA_DW_ins}
\end{table}

\end{landscape}

\subsection{Discussion}
Our numerical experiments provide a clear narrative on the practical challenges and trade-offs involved in solving ~\eqref{SkS_problem_standard_formulation} via  QUBO relaxations. This discussion synthesizes our findings, moving from the performance of exact solvers on small instances to the scalability of iterative heuristic algorithms on larger, more complex graphs.

Our first key finding comes from the comparative analysis of exact methods on  ER  graphs. Although the quadratic penalty relaxation ~\eqref{SkS_QUBO_formulation_one}, the Lagrangian relaxation ~\eqref{SkS_QUBO_Lagrangian_1}, and augmented Lagrangian relaxation ~\eqref{SkS_QUBO_formulation_two}  are all theoretically capable of finding optimal solutions, if we use penalty parameters according to theoretical results established in Section \ref{definitions}, 
they are not computationally equivalent. We are aware that our computational evaluation of exact methods is very limited due to large computational times; therefore, we will talk only about observations that may not be general.

First, the Lagrangian relaxation is limited to specific graph classes where the underlying theoretical assumptions hold ($A_k \le B_k$), making it unsuitable as a general-purpose approach. Moreover, this approach is handy if we know $A_k$ and $B_k$ in prior. Additionally, the efficiency of this approach is very sensitive to the Lagrangian parameter $\lambda$, so the Lagrangian relaxation as an exact approach has very limited practical value.

Second, when comparing the quadratic and augmented Lagrangian relaxations on ER graphs, the results in Fig. \ref{fig: Biqbin_ER} show that the latter is consistently more efficient. For instance, with graphs of size $n=40$, solving the quadratic relaxation required significantly higher computation time and an order of magnitude more Branch \& Bound (B\&B) nodes compared to the augmented Lagrangian relaxation. This performance gap becomes even more pronounced for larger graphs. For seven instances with $n=80$, we did not get results of ~\eqref{SkS_QUBO_formulation_one} within the given time limit of 30 minutes per instance, whereas we could solve ~\eqref{SkS_QUBO_formulation_two} on these instances efficiently. The number of B\&B nodes also supports this observation, as the Augmented Lagrangian relaxation produces a considerably easier search space for the solver.

Overall, based on preliminary computational analysis from Section \ref{sec:exact_solvers}, we conclude that for exact solvers, the Augmented Lagrangian relaxation performs best, offering robustness and computational efficiency compared to the quadratic penalty and the Lagrangian relaxations.
This may be due to the fact that the augmented Lagrangian separately increases the diagonal of the matrix $A$ in the objective function, which makes the resulting QUBO more appropriate for the exact B\&B solver. 

\begin{figure}[H]
   \centering
   \begin{subfigure}[b]{0.46\textwidth}
       \centering
       \includegraphics[width=1.1\linewidth]{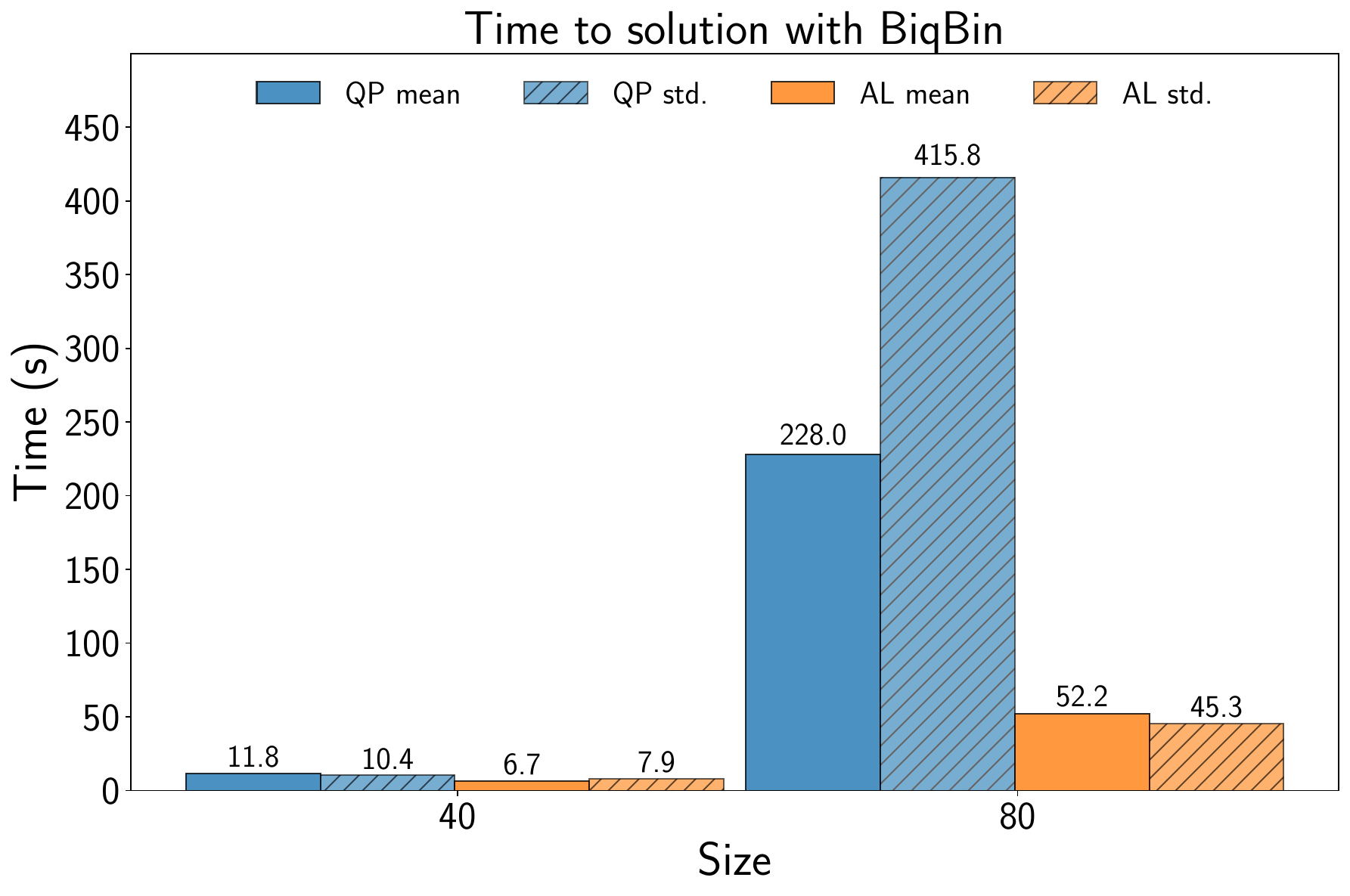}
       % \caption{label 1}
   \end{subfigure}
   \qquad
   \begin{subfigure}[b]{0.46\textwidth}
       \centering
       \includegraphics[width=1.1\linewidth]{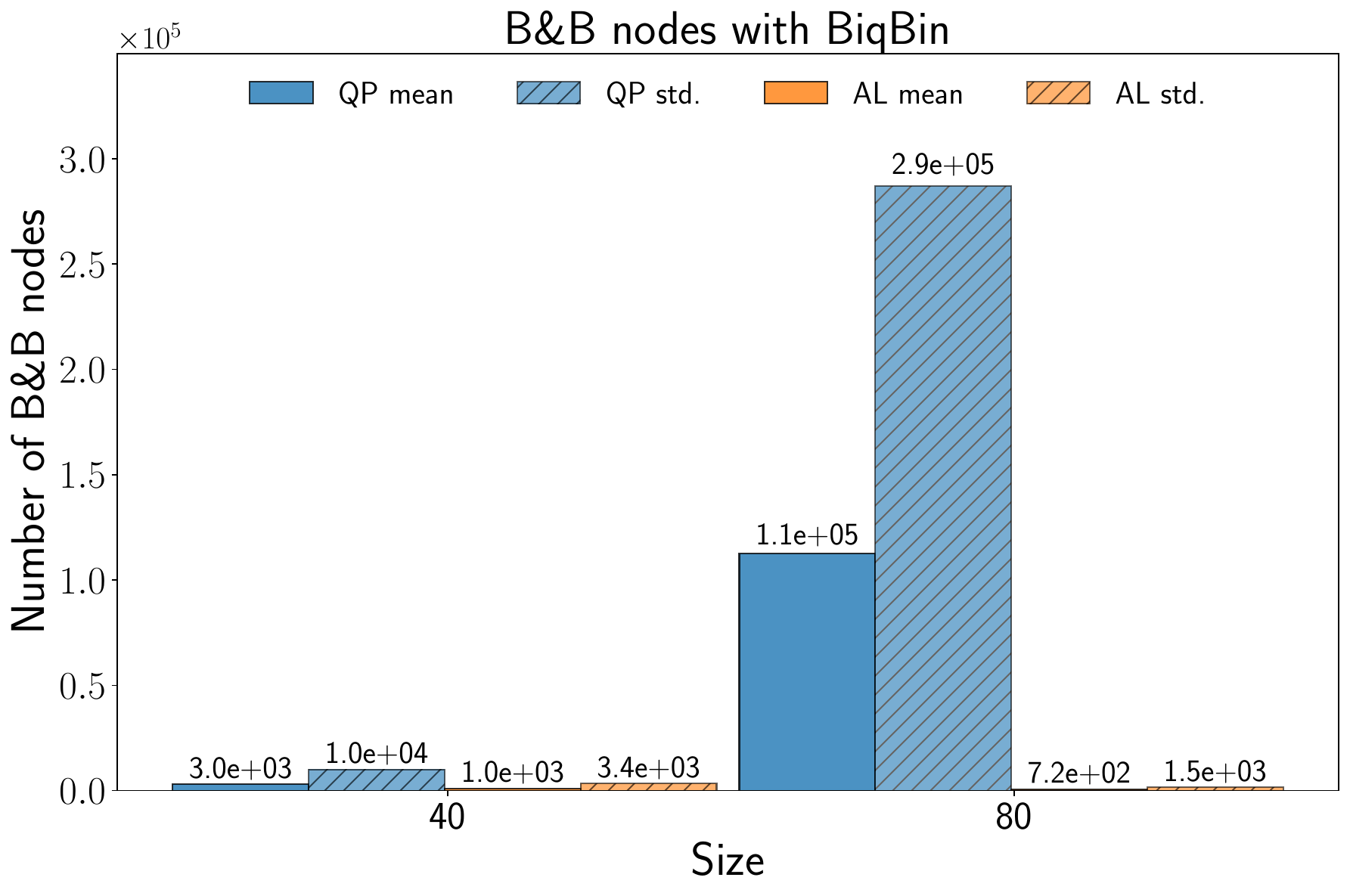}
       % \caption{label 2}
   \end{subfigure}
   \caption{Performance summary of the QP and ALM on  ER  graphs with BiqBin solver.}
   \label{fig: Biqbin_ER}
\end{figure}

The iterative methods were compared on the full set of ER instances and on the  DW instances. Figures \ref{fig: ER_ins} and \ref{fig: DW_ins} depict their overall performance, separately for ER and DW graphs, and for SA and D-Wave QPU solvers.   

From Fig. \ref{fig: ER_ins}, the left-hand diagram shows that we were always able to solve all 270 instances when using the SA solver in all three iterative approaches and that the computed solutions were always feasible, but not always optimal. The best performing method is ALIA, which provided optimal solutions in 269 instances, followed by LRIA, which provided 264 optimal solutions, while QPIA computed optimal solutions in only 5 instances.
We also show the number of instances where greedy post-processing was used. We note that it was only used for the LRIA method, where it was required for 30 instances among all the feasible solutions and led to optimal solutions in 24 instances. The dashed horizontal line indicates the optimal solution obtained by the greedy method only.

The right-hand side of Fig. \ref{fig: ER_ins} shows that when using the D-Wave QPU solver on the ER graphs, we could not obtain an embedding to the D-Wave system 
for all ER instances. The best performing method on the ER instances is LRIA, as we could do embeddings for 234 out of 270 instances, and for all of them, the D-Wave QPU solver returned feasible solutions, but only 28  of them were optimal. Greedy post-processing was performed for one instance only, yielding an optimal solution. QPIA and ALIA required clique embeddings, which are more demanding, so only 135 instances could be embedded. QPIA computed feasible solutions for 32 embedded instances, none of which were optimal, while ALIA returned 93 feasible instances, 9 of which were optimal. 

Fig. \ref{fig: DW_ins} depicts the performance of the iterative methods on DW graphs, separately for the SA solver (left-hand diagram) and the D-Wave QPU solver (right-hand diagram).
When the iterative methods use the SA solver, the performance is better. As with the ER instances, the performance of QPIA is very weak - no optimal solution was computed. The LRIA method returned feasible solutions for all 7 instances, but in three cases, greedy post-processing was needed. Optimal solutions were found for 4 instances, one of them using greedy post-processing.
ALIA returned feasible solutions for all instances, and optimal solutions for 3 of them.

For the D-Wave QPU, we did not obtain optimal solutions with any iterative method. QPIA shows a very weak performance - it could only embed 3 instances, of which no feasible solution was found. LRIA found feasible solutions for all instances, but greedy post-processing was required in 6 instances. ALIA performed slightly better than QPIA: it was able to embed 3 instances and also obtained feasible solutions for them, but no optimal solution.

The box plots on Fig. \ref{fig: ErrorBox_ER} complement Fig.~\ref{fig: ER_ins}. They show the distributions of errors of the feasible solutions obtained by the iterative methods on  ER  instances using SA and D-Wave QPU solvers. For LRIA, we depict the errors of feasible solutions after greedy postprocessing. We write the number of feasible solutions obtained by each solver below the labels of the iterative methods.

We may conclude that the QPIA method has the worst performance, regardless of the solver it uses: with the SA solver, it computes feasible solutions for all instances but with the highest errors, with the D-Wave QPU solver, it computes feasible solutions for only 32 instances, and these also have high errors (the median of the errors is the highest).
 The LRIA method performs slightly better when using the D-Wave QPU, as it computes feasible solutions for 234 out of 270 instances, and all key quantiles for errors are smaller compared to ALIA.
 The ALIA method performs best with the SA solver.
 The main reason for the better performance of the LRIA method with D-Wave QPU is the limitation in embedding, since ALIA and QPIA require clique embedding, while LRIA performs normal embedding based only on the underlying graph, which should be easier since DW graphs are subgraphs of the D-Wave system topology.

\begin{figure}[H]
    \centering
    \begin{subfigure}[b]{0.46\textwidth}
        \centering
        \includegraphics[width=1.1\linewidth]{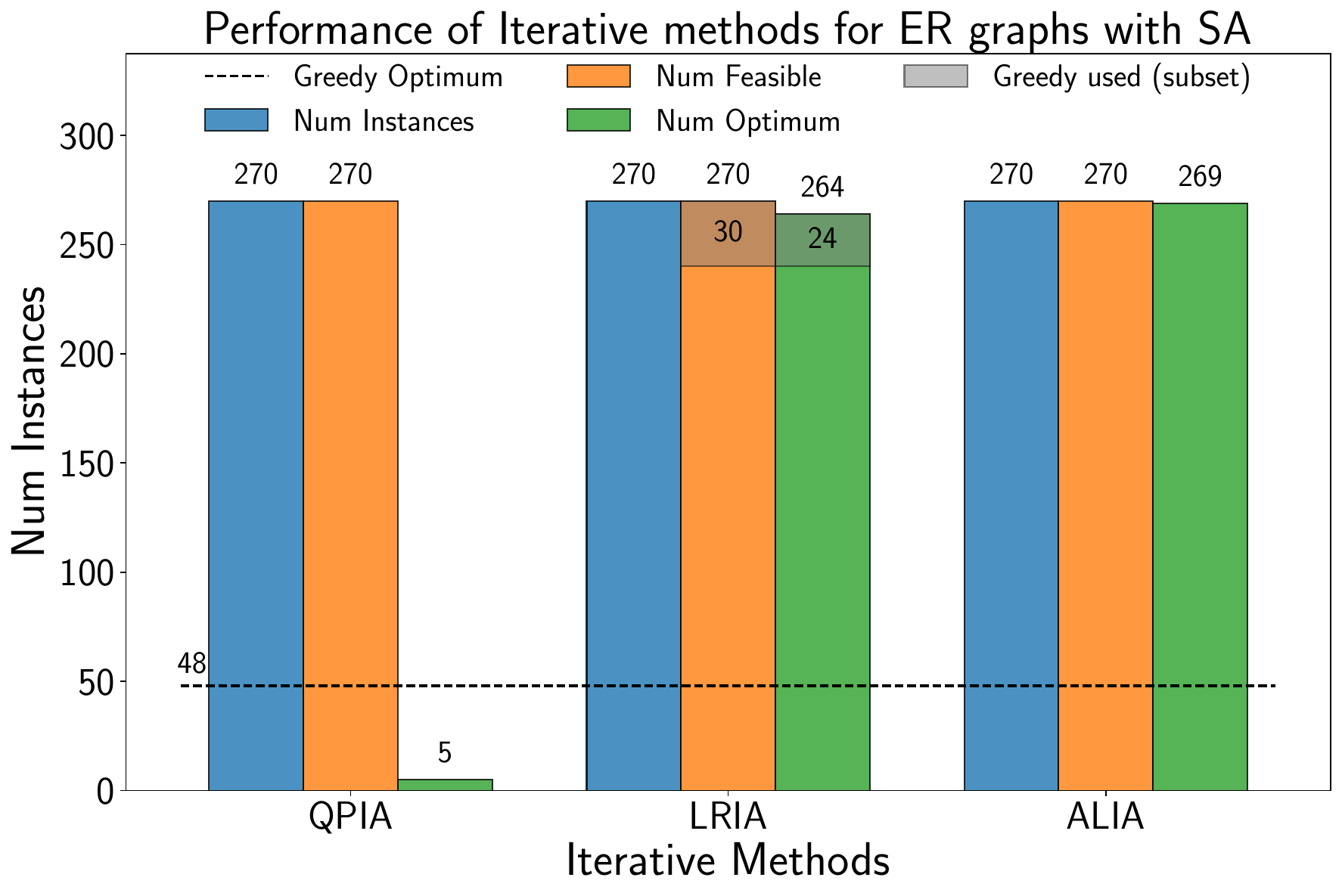}
        % \caption{label 1}
    \end{subfigure}
    \qquad
    \begin{subfigure}[b]{0.46\textwidth}
        \centering
        \includegraphics[width=1.1\linewidth]{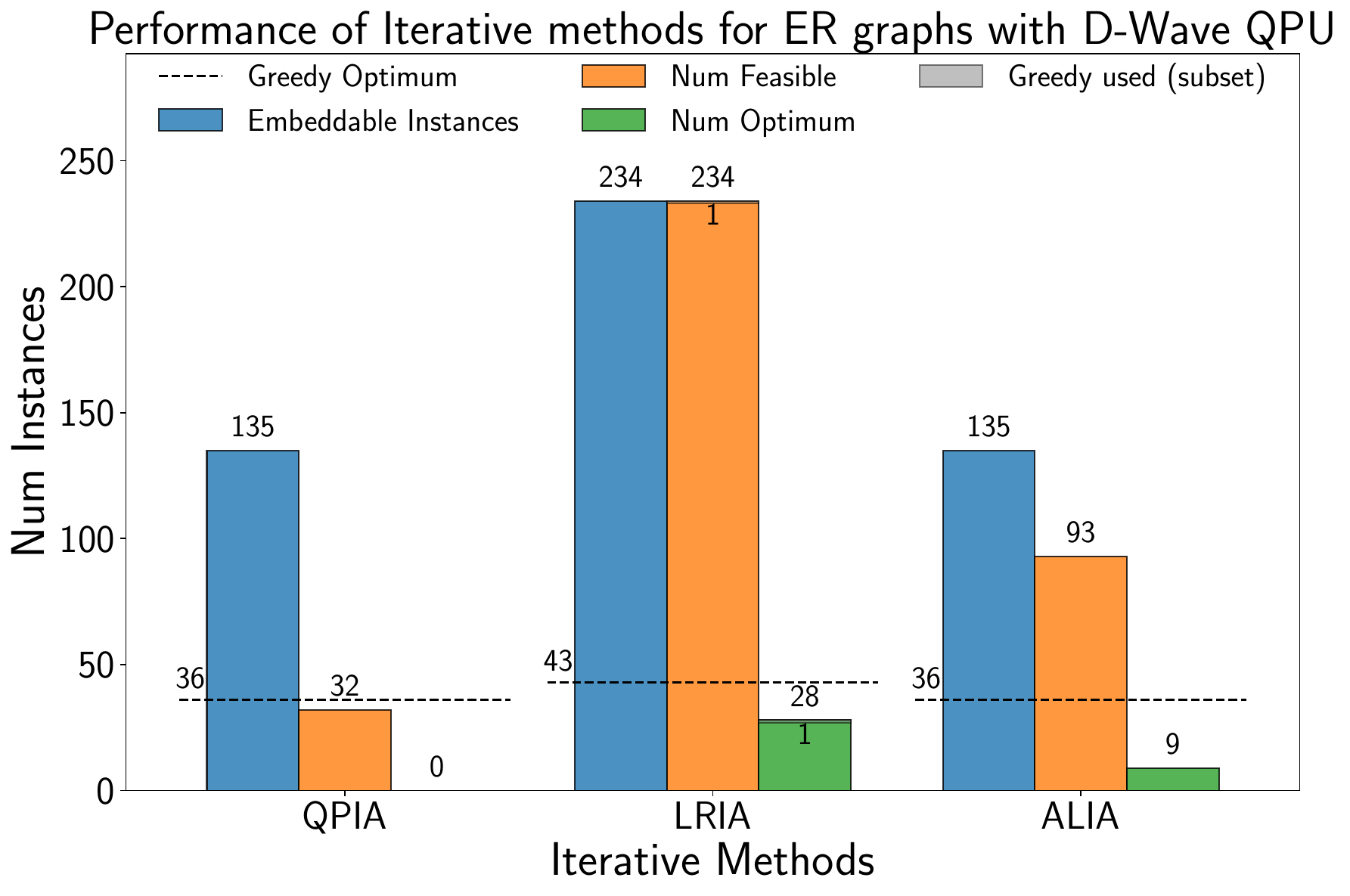}
        % \caption{label 2}
    \end{subfigure}
    \caption{Performance summary of the iterative methods on  ER  graphs.}
    \label{fig: ER_ins}
\end{figure}

\begin{figure}[H]
    \centering
    \begin{subfigure}[b]{0.46\linewidth}
        \centering
        \includegraphics[width=1.1\linewidth]{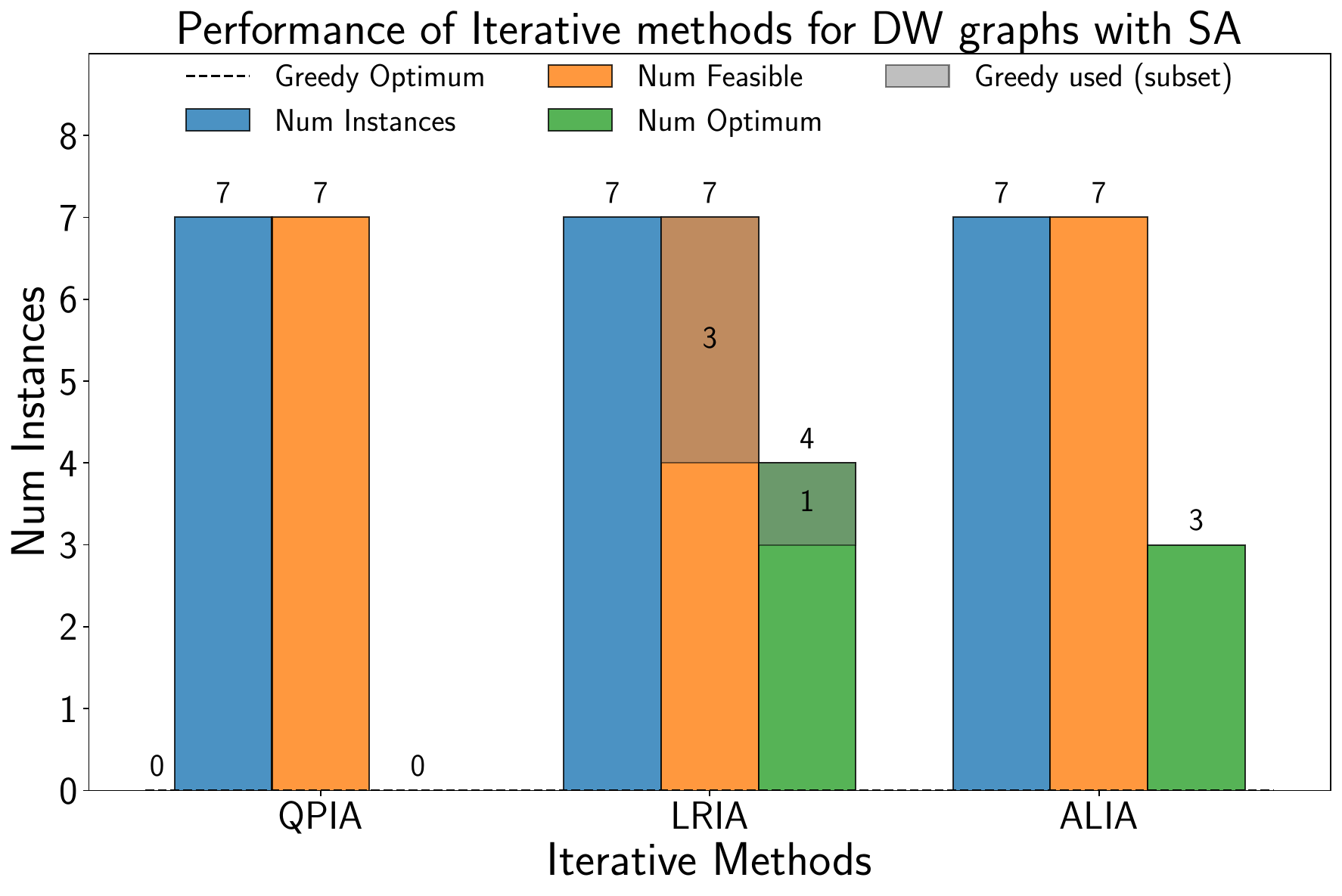}
        % \caption{label 1}
    \end{subfigure}
    \qquad
    \begin{subfigure}[b]{0.46\linewidth}
        \centering
        \includegraphics[width=1.1\linewidth]{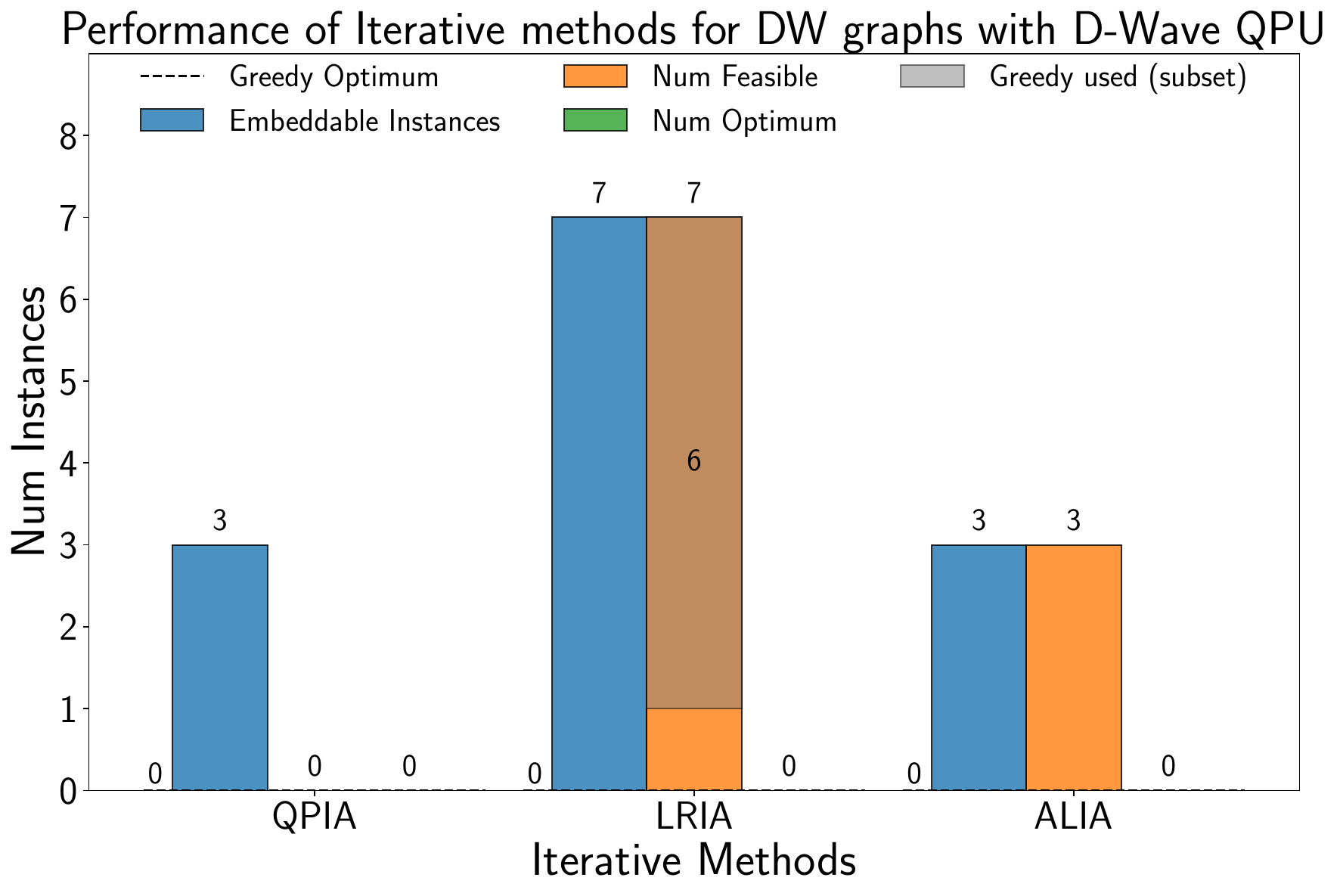}
        % \caption{label 2}
    \end{subfigure}
    \caption{Performance summary of the iterative methods on  DW graphs.}
    \label{fig: DW_ins}
\end{figure}

\begin{figure}[H]
    \centering
    \includegraphics[width=0.7\linewidth]{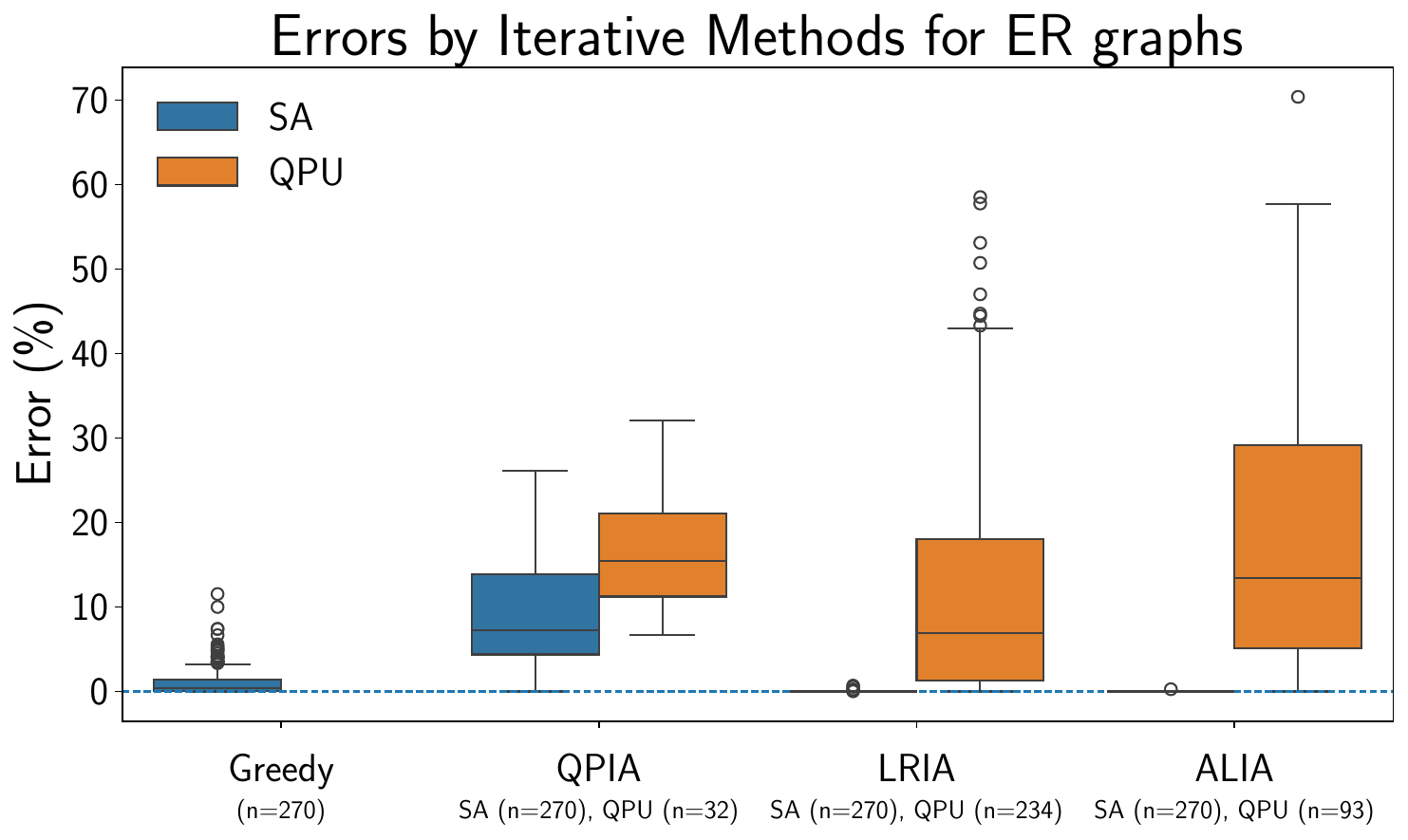}
    \caption{Box-plots for distribution of errors with  iterative methods on  ER  instances  }
    \label{fig: ErrorBox_ER}
\end{figure}

\section{Conclusions}

In this paper, we presented a comprehensive theoretical and computational investigation into solving the sparsest $k$-subgraph (SkS) problem using Quadratic Unconstrained Binary Optimization (QUBO) relaxations. We systematically analyzed three distinct approaches:    quadratic penalty  relaxation (QP),   Lagrangian relaxation (LR), and augmented Lagrangian (AL) relaxation.
Our primary contributions are both theoretical and practical. On the theoretical side, we established rigorous bounds and conditions for the quadratic penalty and Lagrangian parameters ($\mu$ and $\lambda$) that guarantee the exactness of the QUBO relaxations. We also proved that for graphs where the sequence of minimum edge differences $\{\diff_{\ell}\}$ is monotonically increasing, only the Lagrangian multiplier, $\lambda$, is enough to achieve the optimum solution. 

Our computational experiments on three sets of instances (Erd\H{o}s--R\'enyi (ER) graphs, Bipartite graphs, and D-Wave topology graphs) provided clear insights into the practical performance of these methods.
When using BiqBin, an exact QUBO solver, the AL relaxation method was demonstrably superior on the instances that we considered. It consistently solved problem instances with significantly less computational effort, as measured by both time and the number of branch and bound nodes, compared to QP approach, making it the most effective formulation for exact methods.
The LR relaxation is useful only if the underlying graph has a specific property that $\{\diff_{\ell}\}$ is monotonically increasing, which is rare and hard to verify.

When using heuristic solvers for QUBO relaxations, we have to switch to iterative algorithms. We have developed and implemented three iterative algorithms: Quadratic Penalty Iterative Algorithm (QPIA), Lagrangian Relaxation Iterative Algorithm (LRIA) and Augmented Lagrangian Iterative Algorithm (ALIA). They dynamically adjust the quadratic penalty and Lagrangian parameters and use two heuristic solvers (SA and the D-Wave Quantum Annealer) in each iteration. They could solve SkS for larger instances.
We observed a clear difference between the classical Simulated Annealing (SA) solver and the D-Wave QPU.
The QPIA method has the worst performance, regardless of the solver used.
The LRIA method performs slightly better than ALIA when using the D-Wave QPU solver, while the ALIA method with the SA solver performs best and provides optimal solutions for almost all ER instances.
We find that the linear penalty, which is a key feature of LRIA, preserves the sparsity of the problem and thus avoids the costly clique embeddings required when using quadratic penalty terms, so LRIA has the potential to solve larger problems on the quantum annealing hardware.

This work highlights an important lesson for the application of heuristic solvers to constrained optimization problems: the choice of problem formulation is not merely a theoretical exercise, but a crucial factor that has a direct impact on practical performance. This is even more important in the case of the quantum annealer, where hardware connectivity is a limitation.

Despite these findings, several interesting avenues for future research remain open. It would be highly valuable to characterise the classes of graphs for which the $\{\diff_{\ell}\}$ sequence is monotonically increasing, as this would extend the applicability of the Lagrangian relaxation method. A natural next step is to investigate whether our theoretical results and algorithmic approaches can be generalised to other cardinality-constrained optimization problems. Finally, there is significant potential in developing more sophisticated post-processing and graph partitioning techniques, similar to those used for the maximum stable set problem as in \cite{Krpan2024}, to further improve the solution quality and scale to even larger instances.

In our future work, we will also consider how to extend the methods from this paper to quadratic optimization problems with binary variables that have more linear constraints and also some quadratic constraints.

\section*{Aknowledgements}
The authors would like to thank Dr. Jaka Vodeb for conducting the experiments with the D-Wave quantum annealer, and Beno Zupanc for his assistance in carrying out the computational experiments with the BiqBin solver. The authors also thank Miha Rade\v{z} for fruitful discussions.

\section*{Funding} The research work of the first three authors was partially supported by the project Quantum Solver for Hard Binary Quadratic Problems (QBIQ), ID J7-50186,  funded by the Slovenian Research and Innovation Agency (ARIS), the project QEC4QEA, co-funded by EuroHPC JU and the Republic of Slovenia, the Ministry of Higher Education, the project DIGITOP, co-funded by the Republic of Slovenia, the Ministry of Higher Education, 
Science and Innovation, ARIS, and the European Union – 
NextGenerationEU, and by ARIS through the 
Annual work program of Rudolfovo. The fourth author was funded in part  by the Austrian Science Fund (FWF) [10.55776/DOC78]. For the purposes of open access, the authors have applied a CC BY public copyright license to all author-accepted manuscript versions resulting from this submission.

\bibliographystyle{elsarticle-num} 
\bibliography{Ref.bib}

\end{document}